\documentclass[10.5pt, a4paper]{article}
\usepackage{fullpage}
\usepackage{listings}
\usepackage{latexsym, amsfonts, amssymb, amsmath, amsthm, mathrsfs, mathtools, setspace, graphics, graphicx, bbm, float,bigints}
\usepackage{enumerate}
\usepackage{caption}
\usepackage{indentfirst}
\usepackage{framed}
\usepackage{yhmath}
\usepackage[nodayofweek,level]{datetime}

\date{}
\newcommand*{\affaddr}[1]{#1} % No op here. Customize it for different styles.
\newcommand*{\affmark}[1][*]{\textsuperscript{#1}}

\RequirePackage{tikz}
\usetikzlibrary{calc,trees,positioning,arrows,chains,shapes.geometric,%
    decorations.pathreplacing,decorations.pathmorphing,shapes,%
    matrix,shapes.symbols}
\allowdisplaybreaks    

\usepackage{geometry}
\title{Probability measure characterization by $L^{p}$-quantization error function}

\usepackage{hyperref}

\newcommand{\widesim}[2][1.5]{
  \mathrel{\overset{#2}{\scalebox{#1}[1]{$\sim$}}}
}
\title{Convergence Rate of Optimal Quantization  and Application to the Clustering Performance of the Empirical Measure}

\author{
Yating Liu\affmark[1] and  Gilles Pag\`es\affmark[2],\\
\affaddr{\small }\\
\affaddr{\affmark[1] \small CEREMADE, CNRS, UMR 7534, Université Paris-Dauphine,  }\\
\affaddr{\small PSL University, 75016 Paris, France, \small \texttt{ liu@ceremade.dauphine.fr}.}\\
\affaddr{\affmark[2]\small  Sorbonne Université, CNRS, Laboratoire de Probabilit\'es, Statistiques et Mod\'elisations (LPSM),}\\
\affaddr{\small 75252 PARIS Cedex 05, FRANCE,  \texttt{ gilles.pages@sorbonne-universite.fr}.}\\%
}

\begin{document}
\maketitle

\newtheorem*{thm*}{Theorem}
\numberwithin{equation}{section}
\theoremstyle{plain}
\newtheorem{lemma}{Lemma}[section]
\newtheorem{proposition}{Proposition}[section]
\newtheorem{corollary}{Corollary}[section]
\newtheorem{theorem}{Theorem}[section]
\newtheorem{definition}{Definition}[section]
\theoremstyle{remark}
\newtheorem{remark}{Remark}[section]

\renewcommand{\thefootnote}{(\arabic{footnote})}

\maketitle

%\vspace{0.7cm}

\begin{abstract}%   <- trailing '%' for backward compatibility of .sty file
We study the convergence rate of the optimal quantization for a probability measure sequence $(\mu_{n})_{n\in\mathbb{N}^{*}}$ on $\mathbb{R}^{d}$ converging in the Wasserstein distance in two aspects: the first one is the convergence rate of optimal quantizer $x^{(n)}\in(\mathbb{R}^{d})^{K}$ of $\mu_{n}$ at level $K$; the other one is the convergence rate of the distortion function valued at $x^{(n)}$, called the ``performance'' of $x^{(n)}$. Moreover, we also study the mean performance of the optimal quantization for the empirical measure of a distribution $\mu$ with finite second moment but possibly unbounded support. As an application, we show that the mean performance for the empirical measure of the multidimensional normal distribution $\mathcal{N}(m, \Sigma)$ and of distributions with hyper-exponential tails behave like  $\mathcal{O}(\frac{\log n}{\sqrt{n}})$. This extends the results from \cite{biau2008performance} obtained for compactly supported distribution. We also derive an upper bound which is sharper in the quantization level $K$ but suboptimal in $n$ {by applying results in \cite{fournier2015rate}}. 
\end{abstract}

\vspace{0.2cm}
\emph{keywords:} Clustering performance, Convergence rate of optimal quantization, Distortion function, Empirical measure, Optimal quantization.

\vspace{1cm}

\section{Introduction}%(rewrite with connections to the non-parametric learning)}

The $K$-means clustering procedure in the unsupervised learning area  was first introduced by \cite{macqueen1967classification}, which 
%The $K$-means clustering 
consists in partitioning a data set of observations $\{\eta_{1}, ..., \eta_{N}\}\subset \mathbb{R}^{d}$ into $K$ classes $\mathcal{G}_{k}, {1\leq k \leq K}$ with respect to a \textit{cluster center} $x=(x_{1}, ..., x_{K})$ in order to minimize the {\color{black}quadratic} distortion function $\mathcal{D}_{K, \eta}$ defined by
\begin{equation}\label{distortionclustering}
x=(x_{1}, ..., x_{K})\in(\mathbb{R}^{d})^{K}\mapsto \mathcal{D}_{K, \eta}(x)\coloneqq\frac{1}{N}\sum_{n=1}^{N}\min_{k=1, ..., K}d(\eta_{n},x_{k})^{2},
\end{equation}
where $d$ denotes a distance on $\mathbb{R}^{d}$. The classification of the observations $\{\eta_{1}, ..., \eta_{N}\}\subset\mathbb{R}^{d}$ in \cite{macqueen1967classification} can be described as follows %(see [(2.2) and the paragraph which follows (2.2)]) 
\begin{align}\label{clusterdef}
&\mathcal{G}_{1}=\big\{\;\eta_{n}\in\{\eta_{1}, ..., \eta_{N}\}\;:\;d(\eta_{n}, x_{1}) \leq \min_{2\leq j\leq K}d(\eta_{n}, x_{j})\;\big\}\nonumber\\
&\mathcal{G}_{2}=\big\{\;\eta_{n}\in\{\eta_{1}, ..., \eta_{N}\}\;:\;d(\eta_{n}, x_{2}) \leq \min_{1\leq j\leq K, j\neq 2}d(\eta_{n}, x_{j})\;\big\}\setminus \mathcal{G}_{1}\nonumber\\
&\cdots\nonumber\\
&\mathcal{G}_{K}=\big\{\;\eta_{n}\in\{\eta_{1}, ..., \eta_{N}\}\;:\;d(\eta_{n}, x_{K}) \leq \min_{1\leq j\leq K-1}d(\eta_{n}, x_{j})\;\big\}\setminus \big(\mathcal{G}_{K-1}\cup\cdots\cup \mathcal{G}_{1}\big).
\end{align}
If a cluster center $x^{*}=(x^{*}_{1}, ..., x^{*}_{K})$ {\color{black}satisfies} $\mathcal{D}_{K, \eta}(x^{*})=\inf_{y\in(\mathbb{R}^{d})^{K}}\mathcal{D}_{K, \eta}(y)$, we call $x^{*}$ an optimal cluster center (or  \textit{$K$-means}) for the observation $\eta=(\eta_{1}, ..., \eta_{N})$. {\color{black}Such an optimal cluster center always exists but is generally not unique.}%We refer to \cite{macqueen1967classification} and \cite{linder2002principle}[Chapter 4] for more details for the $K$-means clustering.

$K$-means clustering has a close connection with quadratic optimal quantization,  originally developed as a  discretization method for the signal transmission and compression by the Bell laboratories in the 1950s (see \cite{IEEE1982} and \cite{gersho2012vector}). Nowadays, optimal quantization has also become  an efficient tool in numerical probability,  used to provide a discrete representation of a probability distribution. To be more precise, let $\left|\cdot\right|$ denote the Euclidean norm on $\mathbb{R}^{d}$  {induced} by {the canonical} inner product $\langle\cdot{|}\cdot\rangle$ and let $X$ be an $\mathbb{R}^{d}$-valued random variable defined on $(\Omega, \mathcal{F}, \mathbb{P})$ with probability distribution $\mu$ having a finite second moment. The quantization method consists in discretely approximating $\mu$ by using a $K$-tuple $x=(x_{1}, ..., x_{K})\in(\mathbb{R}^{d})^{K}$  and its weight $w=(w_{1}, ..., w_{K})$ as follows, 
\[\mu\simeq\widehat{\mu}^{x}\coloneqq \sum_{k=1}^{K}w_{k}\delta_{x_{k}},\]
where $\delta_{a}$ denotes the Dirac mass at $a$,  the weights $w_{k}$ are computed by $w_{k}=\mu\big(C_{k}(x)\big), k=1, ..., K,$ and $\big(C_{k}(x)\big)_{1\leq k \leq K}$ is a \textit{Vorono\"i partition} induced by $x$, that is, a Borel partition on $\mathbb{R}^{d}$ satisfying 
\begin{equation}\label{voronoidef}
C_{k}(x)\subset V_{k}(x)\coloneqq\big\{\xi\in\mathbb{R}^{d}\;\big|\;\left|\xi-x_{k}\right|=\min_{1\leq j\leq K}\left|\xi-x_{j}\right|\big\}, \;\;k=1, ..., K.
\end{equation}
%{\color{blue}Then }
%We will admit this definition of the weight $w$ in the following discussion. 
The value $K$ in the above description is called the \textit{quantization level} and the $K$-tuple above $x=(x_{1}, ..., x_{K})$ is called a \textit{quantizer} (or \textit{quantization grid,  codebook} in the literature). Moreover, we define the \textit{(quadratic) quantization error function} $e_{K, \mu}$ of $\mu$ (or \textit{of $X$}) at level $K$ by
\begin{equation}\label{defquanerror}
x=(x_{1}, ..., x_{K})\in(\mathbb{R}^{d})^{K}\longmapsto e_{K, \mu}(x)\coloneqq \Big[\int_{\mathbb{R}^{d}}\min_{1\leq k\leq K}\left|\xi-x\right|^{2}\mu(d\xi)\Big]^{1/2}. %=\left\Vert d(X, \Gamma^{x})\right\Vert_{2},
\end{equation}
%where $d(\xi,A)$ is the distance between a point $\xi$ and a set $A\subset\mathbb{R}^{d}$, defined by $d(\xi,A)\coloneqq\min_{a\in A}\left|\xi-a\right|$ and $\Gamma^{x}\coloneqq\{x_{1}, ..., x_{K}\}\subset\mathbb{R}^{d}$ is the quantization grid considered as a discrete point set in $\mathbb{R}^{d}$.
{\color{black}The set $\mathrm{argmin}\,e_{K,\mu}$ is not empty (see e.g. \cite{graf2000foundations}[see Theorem 4.12]) and any element $x^{*}=(x_{1}^{*}, ..., x_{K}^{*})$ in $\mathrm{argmin}\,e_{K,\mu}$ is called a \textit{(quadratic) optimal quantizer} for the probability distribution $\mu$ at level $K$. }%It is generally not unique. }% especially when $d\geq2$ (J'aime pas la partie especially when $$).}
% A quantizer $x^{*}=(x_{1}^{*}, ..., x_{K}^{*})$ is called \textit{(quadratic) optimal} for the probability distribution $\mu$ at level $K$ if $x^{*}\in \mathrm{argmin}\,e_{K,\mu}$. 
Moreover, we call 
\begin{equation}\label{optiquantizationerror}
e_{K,\,\mu}^{*}=\inf_{y=(y_{1}, ..., y_{K})\in(\mathbb{R}^{d})^{K}}e_{K, \mu}(y)
\end{equation}
the \textit{optimal (quadratic) quantization error} (\textit{optimal error}  for short) at level $K$.

The connection between  $K$-means clustering and  quadratic optimal quantization is the following: if the distance $d$ in (\ref{distortionclustering}) and (\ref{clusterdef}) is the Euclidean distance and if we consider the empirical measure $\bar{\mu}_{N}$ of the dataset  $\{\eta_{1}, ..., \eta_{N}\}$ defined by
\begin{equation}\label{empiricalm}
\bar{\mu}_{N}\coloneqq \frac{1}{N}\sum_{n=1}^{N}\delta_{\eta_{n}},
\end{equation} 
%where  
then the distortion function $\mathcal{D}_{K, \eta}$ defined in (\ref{distortionclustering}) is in fact $e_{K, \bar{\mu}_{N}}^{2}$ and 
$\mathrm{argmin}\, \mathcal{D}_{K, \eta}= \mathrm{argmin}\, e_{K, \bar{\mu}_{N}}.$
That is, an optimal quantizer of $\bar{\mu}_{N}$ is in fact an optimal cluster center for the dataset $\{\eta_{1}, ..., \eta_{N}\}$. 

In Figure \ref{figoptimalnormal}, we show an optimal quantizer and its weights for the standard normal distribution $\mathcal{N}\big(\mathbf{0}, \mathbf{I}_{2}\big)$ in $\mathbb{R}^{2}$ at level 60, where $\mathbf{I}_{d}$ denotes the identity matrix of size $d\times d$. %on $\mathbb{R}^{d}$ and 
The color of the cells in the figure represents the weight of each point $x_{k}$ in the quantizer $x=(x_{1}, ..., x_{K})$.  In Figure~\ref{fig2}, we show an optimal cluster center at level $K=20$ for an i.i.d simulated sample $\{\eta_{1}, ..., \eta_{500}\}$ of the $\mathcal{N}(0, \mathbf{I}_{2})$ distribution. 

\begin{figure}[H]
\centering
\begin{minipage}{.5\textwidth}
  \centering
  \includegraphics[width=.9\linewidth]{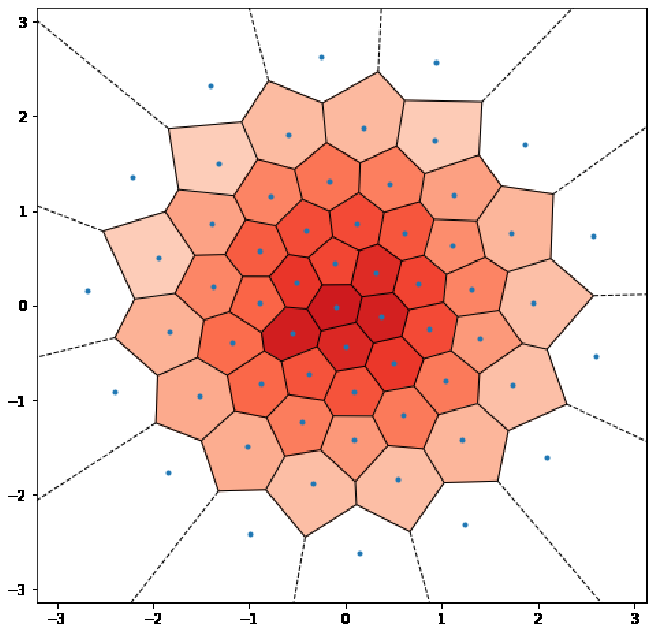}
\caption{An optimal quantizer for  $\mathcal{N}\big(\mathbf{0}, \mathbf{I}_{2}\big)$ \newline at level 60.}
  \label{figoptimalnormal}
\end{minipage}%
\begin{minipage}{.5\textwidth}
  \centering
  \vspace{0.7cm}
  \includegraphics[width=.9\linewidth]{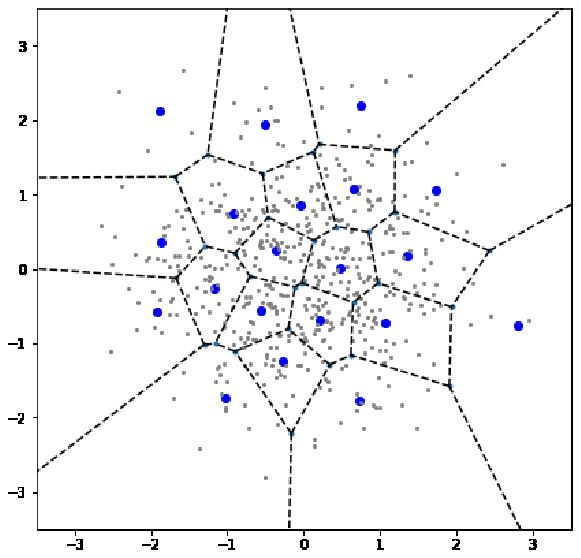}
  \caption{An optimal cluster center (blue points) for an observation $\{\eta_{1}, ..., \eta_{500}\}\widesim{i.i.d}\mathcal{N}(0, \mathbf{I}_{2})$ (grey points).}
  \label{fig2}
\end{minipage}
\end{figure}

For $p\in[1, +\infty)$, let $\mathcal{P}_{p}(\mathbb{R}^{d})$ denote the set of all probability measures on $\mathbb{R}^{d}$ with a finite $p^{th}$-moment. Let $\mu, \nu\in\mathcal{P}_{p}(\mathbb{R}^{d})$ and let $\Pi(\mu,\nu)$ denote the set of all probability measures on $( \mathbb{R}^{d}\times  \mathbb{R}^{d}, Bor( \mathbb{R}^{d})^{\otimes2})$ with marginals $\mu$ and $\nu$, where $Bor(\mathbb{R}^{d})$ denotes the Borel $\sigma$-algebra on $\mathbb{R}^{d}$. For $p\geq1$, the \textit{$L^{p}$-Wasserstein distance} $\mathcal{W}_{p}$ on $\mathcal{P}_{p}(\mathbb{R}^{d})$ is defined by 
\begin{align}\label{defwas2}
\mathcal{W}_{p}(\mu,\nu)&=\Big{(}\inf_{\pi\in\Pi(\mu,\nu)}\int_{\mathbb{R}^{d}\times \mathbb{R}^{d}}\left|x-y\right|^{p}\pi(dx,dy)\Big{)}^{\frac{1}{p}}\nonumber\\
&=\inf\Big{\{}\Big{[}\mathbb{E}\;\left|X-Y\right|^{p}\Big{]}^{\frac{1}{p}},\;  X,Y:(\Omega,\mathcal{A},\mathbb{P})\rightarrow( \mathbb{R}^{d},Bor( \mathbb{R}^{d}))  \;\text{with} \;\mathbb{P}_{X}=\mu, \mathbb{P}_{Y}=\nu\;\Big{\}}.
\end{align}
The space $\mathcal{P}_{p}(\mathbb{R}^{d})$ equipped with the Wasserstein distance $\mathcal{W}_{p}$ is a Polish space, i.e. is separable and complete (see \cite{bolley2008separability}). If $\mu, \nu\in\mathcal{P}_{p}(\mathbb{R}^{d})$, then for any $q\leq p$, $\mathcal{W}_{q}(\mu, \nu)\leq \mathcal{W}_{p}(\mu, \nu).$

With a slight abuse of notation, we define the distortion function for the optimal quantization as follows. 
\begin{definition}[Distortion function]\label{defdistortion}
Let $K\in\mathbb{N}^{*}$ be the quantization level. Let  $\mu\in\mathcal{P}_{2}(\mathbb{R}^{d})$. % and let $\mu$ denote its probability distribution. We assume that $\mu\in\mathcal{P}_{2}(\mathbb{R}^{d})$ and $\mathrm{card}\big(\mathrm{supp}(\mu)\big)\geq K$, 
The \textit{(quadratic) distortion function} $\mathcal{D}_{K,\,\mu}$ of $\mu$ at level $K$ is defined by % on $(\mathbb{R}^{d})^{K}\rightarrow \mathbb{R}_{+}$, by,
\begin{equation}\label{distortiondef}
x=(x_{1}, ...,x_{K})\in(\mathbb{R}^{d})^{K}\longmapsto\mathcal{D}_{K,\,\mu}(x)=\int_{\mathbb{R}^{d}}\min_{1\leq i\leq K}\left|\xi-x_{i}\right|^{2}\mu(d\xi)=e_{K,\mu}^{2}(x).%=\mathbb{E}\min_{1\leq k\leq K}\left| X-x_{k}\right|^{2}=
\end{equation}
\end{definition}

\smallskip
For a fixed (known) probability distribution $\mu$, its optimal quantizers can be computed by several algorithms such as the CLVQ algorithm (see e.g. \cite{pages2015introduction}[Section 3.2]) or the Lloyd I algorithm (see e.g. \cite{MR651807}, \cite{MR651815} and \cite{lloydpagesyu}). However, another situation exists: the probability distribution $\mu$ is unknown but there exists a {\color{black}known} sequence $(\mu_{n})_{n\geq1}$ converging in the Wasserstein distance to $\mu$. 
{\color{black}A typical example is the empirical measure of an i.i.d. $\mu$-distributed sequence random vectors (see (\ref{defempirical}) below). The empirical measure of non i.i.d. random vectors appears for example when dealing with the particle method associated to the McKean-Vlasov equations (see \cite{liu2019optimal}[Section 7.1 and Section 7.5]) or the simulation of the invariant measure of the diffusion process (see \cite{lamberton2002recursive} and \cite{lemaire2005estimation}[Chapter 4]).} 
{\color{black}This leads us to study the consistency and the convergence rate of the optimal quantization for a $\mathcal{W}_{p}$-converging probability distribution sequence $(\mu_{n})_{n\geq1}$. }

\smallskip
%\pagebreak

There exist several studies in the literature.  The consistency of the optimal quantizers was first proved in \cite{pollard1982quantization}. 
%If we consider a probability distribution sequence $\mu_{n}, n\in \mathbb{N}^{*}$ converging in the $L^{2}$-Wasserstein distance to $\mu_{\infty}$, 
%the following theorem shows the consistency of optimal quantizers. % is firstly 
\begin{thm*}[Pollard's Theorem]
Let $\mu_{n}\in\mathcal{P}_{2}(\mathbb{R}^{d}), n\in\mathbb{N}^{*}\cup \{\infty\}$ with $\mathcal{W}_{2}(\mu_{n}, \mu_{\infty})\rightarrow 0$ as $n\rightarrow+\infty$. Assume $\mathrm{card}\big(\mathrm{supp}(\mu_{n})\big)\geq K$, for $n\in\mathbb{N}^{*}\cup\{+\infty\}$. For $n\geq1$, let $x^{(n)}=\big(x_{1}^{(n)},... , x_{K}^{(n)}\big)$ be a $K$-optimal quantizer for $\mu_{n}$, then the quantizer sequence $(x^{(n)})_{n\geq1}$ is bounded in $\mathbb{R}^{d}$ and any limiting point of $(x^{(n)})_{n\geq1}$, denoted by $x^{(\infty)}$, is an optimal quantizer of $\mu_{\infty}$. \footnote{In \cite{pollard1982quantization}[see Theorem 9], the author used \[\mu_{K}\in\mathcal{P}(K)\coloneqq\Big\{\nu\in\mathcal{P}_{2}(\mathbb{R}^{d}) \;\text{such that}\; \text{card\big(supp($\nu$)\big)}\leq K \Big\}\]
to represent a ``quantizer'' at level $K$.  Such a quantizer $\mu_{K}$ is called ``quadratic optimal'' for a probability measure $\mu$ if $\mathcal{W}_{2}(\mu_{K}, \mu)=e_{K,\,\mu}^{*}$.
We propose an alternative proof in Appendix A by using the usual representation of the quantizer $x\in(\mathbb{R}^{d})^{K}$  but  still call this theorem ``Pollard's Theorem''.}
\end{thm*}

{\color{black}
Let $\mu_{n}\in\mathcal{P}_{2}(\mathbb{R}^{d}),\, n\in\mathbb{N}\cup \{\infty\}$ with $\mathcal{W}_{2}(\mu_{n}, \mu_{\infty})\rightarrow0$ as $n\rightarrow+\infty$. Let $x^{(n)}$ denote an optimal quantiser of $\mu_{n}$.  There are two ways to study the convergence rate of the optimal quantizers. The first way is to directly evaluate the distance between $x^{(n)}$ and $\mathrm{argmin}\,\mathcal{D}_{K, \mu_{\infty}}$. The second way is called \textit{the quantization performance}, defined by 
\begin{equation}\label{quantiperfodef}
\mathcal{D}_{K, \mu_{\infty}}(x^{(n)})-\inf_{x\in(\mathbb{R}^{d})^{K}}\mathcal{D}_{K, \mu_{\infty}}(x).
\end{equation}
This quantity describes the distance between the optimal  error of $\mu_{\infty}$ and the quantization error of $x^{(n)}$ considered as a quantizer of $\mu_{\infty}$ (even $x^{(n)}$ is obviously not ``optimal'' for $\mu_{\infty}$). Several results of convergence rate exist in the framework of the empirical measure. 
}
%In the framework of the empirical measure, that is, 
Let $X_{1}(\omega), ..., X_{n}(\omega), ...$ be i.i.d random vectors with probability distribution $\mu\in\mathcal{P}_{2}(\mathbb{R}^{d})$ and let 
\begin{equation}\label{defempirical}
\mu^{\omega}_{n}\coloneqq\frac{1}{n}\sum_{i=1}^{n}\delta_{X_{i}(\omega)}
\end{equation}
be the empirical measure of $\mu$.  {The almost sure convergence of $\mathcal{W}_{2}(\mu_{n}^{\omega}, \mu)$ has been proved in \cite{pollard1982quantization}[Theorem 7].} Let $x^{(n), \omega}$ denotes an optimal quantizer of $\mu_{n}^{\omega}$ at level $K$. 
In \cite{pollard1982central}, the author has proved that if $\mu$ has a unique optimal quantizer $x$ at level $K$, then the convergence rate (convergence in distribution)  of $\left|x^{(n),\omega}-x\right|$ is $\mathcal{O}(n^{-1/2})$ under appropriate conditions. Moreover, if $\mu$ has a support contained in $B(0, R)$, where $B(0, R)$ denotes the ball in $\mathbb{R}^{d}$ centered at $0$ with radius $R$, an upper bound of the mean performance has been proved in \cite{biau2008performance}, shown as follows,
\[\mathbb{E}\,\mathcal{D}_{K,\,\mu}(x^{(n), \omega}) - \inf_{x\in(\mathbb{R}^{d})^{K}}\mathcal{D}_{K,\,\mu}(x)\leq \frac{12K\cdot R^{2}}{\sqrt{n}}.\]

In this paper, we extend the convergence results in %rate of the optimal quantization in a larger framework than 
\cite{pollard1982central} and in \cite{biau2008performance}
 in two perspectives: first, we give an upper bound of the quantization performance 
\begin{equation}\label{quantiperfodef2}\mathcal{D}_{K, \mu_{\infty}}(x^{(n)})-\inf_{x\in(\mathbb{R}^{d})^{K}}\mathcal{D}_{K, \mu_{\infty}}(x).\end{equation}
and {\color{black}that of related optimal quantizers }%$x^{(n)}$ and $x^{(\infty)}$} %(j'aime pas preciser $x^{(n)}$ et $x^{(\infty)}$ ici car ce qu'on discute est $d(x^{(n)}, G_{K}(\mu_\infty))$. On n'est pas sure de la convergence de $x^{(n)}$)} 
for any probability distribution sequence $(\mu_{n})_{n\geq1}$ converging in the Wasserstein distance. 
Then, we generalize the clustering performance results in \cite{biau2008performance} to empirical measures in $\mathcal{P}_{2}(\mathbb{R}^{d})$ possibly having an unbounded support.

\smallskip
{Our main results are as follows. We obtain in Section \ref{section2} a non-asymptotic upper bound for the quantization performance:
for every $n\in\mathbb{N}^{*}$, 
\begin{equation}\label{perfoo}
\mathcal{D}_{K,\, \mu_{\infty}}(x^{(n)})-\inf_{x\in(\mathbb{R}^{d})^{K}}\mathcal{D}_{K,\,  \mu_{\infty}}(x) \leq4e_{K, \mu_{\infty}}^{*}\mathcal{W}_{2}(\mu_{n}, \mu_{\infty})+4\mathcal{W}_{2}^{2}(\mu_{n}, \mu_{\infty}).
\end{equation}
Moreover, if $\mathcal{D}_{K, \mu_{\infty}}$ is twice differentiable at 
\begin{equation}\label{fklabel}
F_{K}\coloneqq \big\{ x=(x_{1}, ..., x_{K})\in(\mathbb{R}^{d})^{K}\;\big|\; x_{i}\neq x_{j}, \text{ if } i\neq j\big\}
\end{equation}
and if the Hessian matrix $H_{\mathcal{D}_{K, \mu_{\infty}}}$ of $\mathcal{D}_{K, \mu_{\infty}}$ is positive definite in the neighboorhood of every optimal quantizer $x^{(\infty)}\in G_{K}(\mu_{\infty})$ having the eigenvalues lower bounded by a $\lambda^{*}>0$, then, for $n$ large enough, 
\[d\big(x^{(n)}, G_{K}(\mu_{\infty})\big)^{2}\leq \frac{8}{\lambda^{*}}e_{K, \mu_{\infty}}^{*}\cdot\mathcal{W}_{2}(\mu_{n},\mu_{\infty})+\frac{8}{\lambda^{*}}\cdot\mathcal{W}^{2}_{2}(\mu_{n},\mu_{\infty}),\] 
where $d(\xi,A)\coloneqq\min_{a\in A}\left|\xi-a\right|$ denotes the distance between a point {\color{black}$\xi\in \mathbb{R}^{d}$ and a set $A\subset \mathbb{R}^{d}$}.

\smallskip

Several criterions for the positive definiteness of the Hessian matrix $H_{\mathcal{D}_{K, \mu}}$ of the distortion function $\mathcal{D}_{K, \mu}$ are established in Section \ref{positivedef}. %If $\mu\in\mathcal{P}_{2}(\mathbb{R}^{d})$ with $\mathrm{card}\big(\mathrm{supp}(\mu)\big)\geq K$ and if $\mu$ is absolutely continuous with respect to the Lebesgue measure having a continuous density function $f$, 
We show in Section \ref{continuityhg} the conditions under which the distortion function $\mathcal{D}_{K, \mu}$ is twice differentiable in every $x\in F_{K}$ and give the exact formula of the Hessian matrix $H_{\mathcal{D}_{K, \mu}}$. Moreover, we also discuss several sufficient and necessary conditions for the positive definiteness of the Hessian matrix in dimension $d\geq2$ and in dimension 1. 

}

In Section \ref{section3}, we give two upper bounds for the \textit{clustering performance}  $\mathbb{E}\,\mathcal{D}_{K,\,\mu}(x^{(n), \, \omega}) - \inf_{x\in(\mathbb{R}^{d})^{K}}\mathcal{D}_{K,\,\mu}(x)$, where $x^{(n),\, \omega}$ is an optimal quantizer of $\mu_{n}^{\omega}$ defined in (\ref{defempirical}).  %which is also called the \textit{clustering performance} in the field of unsupervised learning. 
If $\mu\in\mathcal{P}_{q}(\mathbb{R}^{d})$ for some $q>2$, a first upper bound is established in Proposition \ref{convrateemprical}%, shown in the following formula:  
{\begin{align}
\mathbb{E}\,\mathcal{D}_{K,\,\mu}&(x^{(n),\omega})-\inf_{x\in(\mathbb{R}^{d})^{K}}\mathcal{D}_{K,\,\mu}(x)\nonumber\\
&\leq C_{d, q,\mu, K}\times \begin{cases}
n^{-1/4}+n^{-(q-2)/2q}& \mathrm{if}\:d<4\:\mathrm{and}\:q\neq4\\
n^{-1/4}\big(\log(1+n)\big)^{1/2}+n^{-(q-2)/2q}  & \mathrm{if}\:d=4\:\mathrm{and}\:q\neq4\\
n^{-1/d}+n^{-(q-2)/2q} & \mathrm{if}\:d>4\:\mathrm{and}\:q\neq d/(d-2)\nonumber
\end{cases},
\end{align}
}
\hspace{-0.1cm}where %{$q\in(2, 2+\eta)$} and % $q\in(2, 2+\varepsilon)$ and 
$C_{d, q,\mu, K}$ is a constant depending on $d, q, \mu$ and the quantization level $K$. This result is 
a direct application of the non-asymptotic upper bound (\ref{perfoo}) combined with results in \cite{fournier2015rate} about the mean convergence rate of the empirical measure for the Wasserstein distance.
If $d\geq4$ and $q>\frac{2d}{d-2}$, this constant $C_{d, q,\mu, K}$ %in the above inequality depends on $\mu, d, \eta, K$ and especially, it 
is roughly decreasing as $K^{-1/d}$ (see Remark \ref{remC1d}). This upper bound is sharper in $K$ compared with the upper bound (\ref{resultintro}) below, although it suffers from the curse of dimensionality.

Meanwhile, we establish another upper bound for the clustering performance in Theorem \ref{performance}, which is sharper in $n$ %free from the curse of dimensionality 
but increasing faster than linearly in $K$.  This upper bound is
\begin{equation}\label{resultintro}
\mathbb{E}\,\mathcal{D}_{K,\,\mu}(x^{(n), \omega}) - \inf_{x\in(\mathbb{R}^{d})^{K}}\mathcal{D}_{K,\,\mu}(x)\leq \frac{2K}{\sqrt{n}} \Big[ r_{2n}^{2}+\rho_{K}(\mu)^{2}+2r_{1}\big(r_{2n}+\rho_{K}(\mu)\big)\Big],
\end{equation}
where $r_{n}\coloneqq \big\Vert\max_{1\leq i \leq n}\left| X_{i}\right|\big\Vert_{2}$ and $\rho_{K}(\mu)$ is the maximum radius of optimal quantizers for $\mu$, defined by
\begin{equation}\label{maxradius}
\rho_{K}(\mu)\coloneqq \max\Big\{\max_{1\leq k \leq K}\left| x^{*}_{k}\right|,\;\; (x^{*}_{1}, ..., x^{*}_{K}) \text{ is an optimal quantizer of $\mu$ at level $K$}\Big\}.
\end{equation}
In particular, we give a precise upper bound for $\mu=\mathcal{N}(m, \Sigma)$,
 the multidimensionnal normal distribution
\begin{equation}\label{performemprical}
\mathbb{E}\,\mathcal{D}_{K,\,\mu}(x^{(n), \omega}) - \inf_{x\in(\mathbb{R}^{d})^{K}}\mathcal{D}_{K,\,\mu} (x)\leq C_{\mu}\cdot \frac{2K}{\sqrt{n}}\Big[ 1+ \log n+\gamma_{K}\log K\big(1+\frac{2}{d}\big)\Big],
\end{equation}
where $\limsup_{K}\gamma_{K}=1$ and $C_{\mu}=12\cdot\Big[1\vee \log\big(2\int_{\mathbb{R}^{d}}\text{exp}(\frac{1}{4}\left|\xi\right|^{4})\mu(d\xi)\big)\Big]$. If $\mu=\mathcal{N}(0, \mathbf{I}_{d})$, $C_{\mu}=12(1+\frac{d}{2})\cdot\log 2$. 

\smallskip
We start our discussion with a brief review on the properties of  optimal quantization.

\subsection{Properties of {the} Optimal Quantization}\label{propoptimal}

\noindent Let 
$G_{K}(\mu)=\mathrm{argmin}\,\mathcal{D}_{K, \mu}$ denote the set of all optimal quantizers at level $K$ of $\mu$ and let $e^{*}_{K, \mu}$ denote the optimal quantization error of $\mu$ defined in (\ref{optiquantizationerror}). %The properties below recall some classical background on optimal quantization of probability measure. 
\begin{proposition}\label{relation}
Let $K\in\mathbb{N}^{*}$. Let $\mu\in\mathcal{P}_{2}(\mathbb{R}^{d})$ and $\mathrm{card}\big(\mathrm{supp}(\mu)\big)\geq K$.
\begin{enumerate}[(i)]
\item %\textit{(Decreasing of $K\mapsto e_{K,\,\mu}^{*}$)}
 If $K\geq2$, then $e_{K,\,\mu}^{*}< e_{K-1,\,\mu}^{*}$.
\item \textit{(Existence and boundedness of optimal quantizers)} The set $G_{K}(\mu)$
%\[G_{K}'(\mu)\coloneqq\big\{\Gamma^{x}=\{x_{1}, ..., x_{K}\}\;\big|\;x=(x_{1}, ..., x_{K})\in\mathrm{argmin}\,\mathcal{D}_{K, \mu}\big\} \] 
is nonempty and compact  so that $\rho_{K}(\mu)$ defined in (\ref{maxradius}) is finite for any fixed $K$.  Moreover, if $x=(x_{1}, ..., x_{K})$ is an optimal quantizer of $\mu$, then $x\in F_{K}$, where $F_{K}$ is defined in (\ref{fklabel}). %at level $K$ and if we denote by $\Gamma^{x}=\{x_{1}, ..., x_{K}\}$,
%then $\mathrm{card}(\Gamma^{x})=K$.
%$\Gamma^{*}\subset \mathbb{R}^{d}$ is an optimal quantizer of $\mu$, then $\mathrm{card}(\Gamma^{*})=K$. In particular, if $\Gamma^{*}=\{x_{1}, ..., x_{K}\}$, then $x^{\Gamma^{*}}\coloneqq(x_{1}, ..., x_{K})\in\mathrm{argmin}\,\mathcal{D}_{K,\,\mu}$ and vice versa.
\item If the support of $\mu$, denoted by $\mathrm{supp}(\mu)$, is a compact, %  and if the norm $\left|\cdot\right|$  on $\mathbb{R}^{d}$ is Euclidean, 
then for every optimal quantizer $\,x=(x_{1}, ..., x_{K})\in G_{K}(\mu)$, its elements $x_{k}, 1\leq k\leq K$ are 
%all the optimal quantizers $\Gamma^{*}=\{x_{1}, ..., x_{K}\}$ are 
contained in the closure of convex hull of $\mathrm{supp}(\mu)$, denoted by $\mathcal{H}_{\mu}\coloneqq \overline{\mathrm{conv}\big(\mathrm{supp}(\mu)\big)}$. 
\end{enumerate}
\end{proposition}
For the proof of Proposition \ref{relation}-(i) and (ii), we refer to \cite{graf2000foundations}[see Theorem 4.12] and  for the proof of (iii) to Appendix B. 
\begin{thm*}(Non-asymptotic Zador's Theorem, see \cite{luschgy2008functional} and \cite{pages2018numerical}[Theorem 5.2]) Let $\eta>0$. If $\mu\in\mathcal{P}_{2+\eta}(\mathbb{R}^{d})$, then for every quantization level $K$, there exists a constant $C_{d, \eta}\in(0, +\infty)$ which depends only on $d$ and $\eta$ such that 
\begin{equation}\label{nonasymptoticzador}
e_{K, \mu}^{*}\leq C_{d, \eta}\cdot\sigma_{2+\eta}(\mu) K^{-1/d},
\end{equation}
where for $r\in(0, +\infty)$, $\sigma_{r}(\mu)=\min_{a\in\mathbb{R}^{d}}\big[\int_{\mathbb{R}^{d}}\left|\xi-a\right|^{r}\mu(d\xi)\big]^{1/r}$. 
\end{thm*}

When $\mu$ has an unbounded support, we know from  \cite{pages2012asymptotics} that $\lim_{K}\rho_{K}(\mu)=+\infty$. The same paper also gives an asymptotic upper bound of $\rho_{K}$ when $\mu$ has a polynomial tail or a hyper-exponential tail. 

\begin{thm*}(\cite{pages2012asymptotics}[see Theorem 1.2])
Let $\mu\in\mathcal{P}_{p}(\mathbb{R}^{d})$ be absolutely continuous with respect to the Lebesgue measure $\lambda_{d}$ on $\mathbb{R}^{d}$ and let $f$ denote its density function.  
\begin{enumerate}[(i)]
\item \emph{Polynomial tail.} For $p\geq2$, if $\mu$ has a $c$-th polynomial tail with $c>d+p$ in the sense that there exists $\tau>0, \beta\in\mathbb{R}$ and $A>0$ such that $\forall \xi\in\mathbb{R}^{d}, \left|\xi\right|\geq A \Longrightarrow\, f(\xi)=\frac{\tau}{\left|\xi\right|^{c}}(\log\left|\xi\right|)^{\beta}$, then 
\begin{equation}\label{polynomialtail}
 \lim_{K}\frac{\log \rho_{K}}{\log K}=\frac{p+d}{d(c-p-d)}.
 \end{equation}
\item \emph{Hyper-exponential tail.} If $\mu$ has a $(\vartheta, \kappa)$-hyper-exponential tail in the sense that 
there exists $\tau>0, \kappa,\vartheta>0, c>-d$ and $A>0$ such that $\forall \xi\in\mathbb{R}^{d}, \left|\xi\right|\geq A\Longrightarrow\, f(\xi)=\tau \left| \xi\right|^{c}e^{-\vartheta\left|\xi\right|^{\kappa}}$, then
\begin{equation}\label{upperbound}
\limsup_{K}\frac{\rho_{K}}{(\log K)^{1/\kappa}}\leq 2\vartheta^{-1/\kappa}\Big(1+\frac{2}{d}\Big)^{1/\kappa}.
\end{equation}
Furthermore, if $d=1$, $\lim_{K}\frac{\rho_{K}}{(\log K)^{1/\kappa}}=\big(\frac{3}{\vartheta}\big)^{1/\kappa}$. 
\end{enumerate}
\end{thm*}

We give now the definition of the \textit{radially controlled} distribution, which will be useful to  \textit{control} the convergence rate of the density function $f(x)$ to 0 when $x$ converges in every direction to infinity.
\begin{definition}\label{deftail}
Let $\mu\in\mathcal{P}_{2}(\mathbb{R}^{d})$ be absolutely continuous with respect to the Lebesgue measure $\lambda_{d}$ on $\mathbb{R}^{d}$  having a continuous density function $f$.  We call  $\mu$ is {\color{black}\textit{$k$-radially controlled on $\mathbb{R}^d$}}
%has a \textit{$k$-th radial-controlled tail} 
if there exists $A>0$ and a {continuous  non-increasing} function $g: \mathbb{R}_{+}\rightarrow \mathbb{R}_{+}$ such that 
\[\forall \xi\in\mathbb{R}^{d}, \left|\xi\right|\geq A,\;\;\;\;\;\;\;\;\; f(\xi)\leq g(\left|\xi\right|) \;\text{and}\;\int_{\mathbb{R}_{+}}x^{d-1+k}g(x)dx<+\infty.\] 
\end{definition}
%The purpose of this definition is to \textit{control} the convergence rate of the density function $f(x)$ to 0 when $x$ converges in every direction to infinity. 
\noindent Note that the $c$-th polynomial tail with {\color{black}$c>k+d$} and the hyper-exponential tail are sufficient conditions to satisfy the {\color{black}$k$}-radially controlled assumption. A typical example of hyper-exponential tail is the multidimensional normal distribution $\mathcal{N}(m, \Sigma)$.

\smallskip

For $\mu,\nu\in \mathcal{P}_{2}(\mathbb{R}^{d})$ and % if we denote by $\mathcal{D}_{K,\,\mu}$ the distortion function of $\mu$ and $\mathcal{D}_{K,\,\nu}$ the distortion function of $\nu$. Then, 
for every $K\in\mathbb{N}^{*}$, we have
\begin{equation}\label{control}
\left\Vert e_{K, \mu}-e_{K, \nu}\right\Vert_{\sup}%=\left\Vert\mathcal{D}_{K,\,\mu}^{1/2}-\mathcal{D}_{K,\,\nu}^{1/2}\right\Vert_{\sup}
\coloneqq\sup_{x\in(\mathbb{R}^{d})^{K}}\left|e_{K,\,\mu}(x)-e_{K,\,\nu}(x)\right|\leq\mathcal{W}_{2}(\mu, \nu),
\end{equation}
by a simple application of the triangle inequality for the $L^{2}-$norm (see e.g. \cite{graf2000foundations} Formula (4.4) and Lemma 3.4).
Hence, if $(\mu_{n})_{n\geq1}$ is a sequence in $\mathcal{P}_{2}(\mathbb{R}^{d})$ converging for the $\mathcal{W}_{2}$-distance to $\mu_{\infty}\in\mathcal{P}_{2}(\mathbb{R}^{d})$, then for every $K\in\mathbb{N}^{*}$,
\begin{equation}\label{06}
\left\Vert e_{K,\,\mu_{n}}-e_{K,\,\mu_{\infty}}\right\Vert_{\sup}\leq\mathcal{W}_{2}(\mu_{n},\mu_{\infty})\xrightarrow{n\rightarrow+\infty}0.
\end{equation}

\section{General Case}\label{section2}

\noindent In this section, we first establish in Theorem \ref{cvgratedistortion} a non-asymptotic upper bound  of the quantization performance $\mathcal{D}_{K, \,\mu_{\infty}}(x^{(n)})-\inf_{x\in(\mathbb{R}^{d})^{K}}\mathcal{D}_{K, \mu_{\infty}}(x)$. Then we discuss the convergence rate of the optimal quantizer sequence in Theorem \ref{cvgrate}.

{

\begin{theorem}[Non-asymptotic upper bound for the quantization performance]\label{cvgratedistortion} 

Let $K\in\mathbb{N}^{*}$ be the quantization level.  For every $n\in\mathbb{N}^{*}\cup\{\infty\}$, let $\mu_{n}\in\mathcal{P}_{2}(\mathbb{R}^{d})$ with $\mathrm{card}\big(\mathrm{supp}(\mu_{n})\big)\geq K$. Assume that $\mathcal{W}_{2}(\mu_{n}, \mu_{\infty})\rightarrow0$ as $n\rightarrow+\infty$.  For every $n\in\mathbb{N}^{*}$, let $x^{(n)}$ be an optimal quantizer of $\mu_{n}$. Then 
\begin{align}
%&(i)\quad  e_{K, \mu_{\infty}}(x^{(n)})-e_{K, \mu_{\infty}}^{*}\leq 2\mathcal{W}_{2}(\mu_{n}, \mu_{\infty})&\nonumber\\
&\quad \mathcal{D}_{K, \mu_{\infty}}(x^{(n)})-\inf_{x\in(\mathbb{R}^{d})^{K}}\mathcal{D}_{K, \mu_{\infty}}(x) \leq4e_{K, \mu_{\infty}}^{*}\mathcal{W}_{2}(\mu_{n}, \mu_{\infty})+4\mathcal{W}_{2}^{2}(\mu_{n}, \mu_{\infty}),&\nonumber
\end{align}
where $e_{K, \mu_{\infty}}^{*}$ is the optimal error of $\mu_{\infty}$ at level $K$ defined in (\ref{optiquantizationerror}). 
\end{theorem}

\begin{proof}[Proof of Theorem \ref{cvgratedistortion}]

Let $x^{(\infty)}$ be an optimal quantizer of $\mu_{\infty}$. Remark that here we do not need that $x^{(\infty)}$ is the limit of $x^{(n)}$. 
{\color{black}First, we have  (see e.g. Corollary 4.1 in \cite{linder2002principle})
\begin{align}\label{perforerror}
e_{K, \mu_{\infty}}(x^{(n)})-e_{K, \mu_{\infty}}^{*} &= e_{K, \mu_{\infty}}(x^{(n)})-e_{K, \mu_{n}}(x^{(n)})+e_{K, \mu_{n}}(x^{(n)})-e_{K, \mu_{\infty}}(x^{(\infty)})\nonumber\\
&\leq 2\left\Vert e_{K, \mu_{\infty}}- e_{K, \mu_{n}} \right\Vert_{\sup}\leq 2\mathcal{W}_{2}(\mu_{n}, \mu_{\infty}),
\end{align}
where the first inequality is due to the fact that for any $\mu, \nu\in\mathcal{P}_{2}(\mathbb{R}^{d})$ with respective $K$-level optimal quantizers $x^{\mu}$ and $x^{\nu}$, if $e_{K, \mu}(x^{\mu})\geq e_{K, \nu}(x^{\nu})$, we have
\begin{align}
\left|e_{K, \mu}(x^{\mu})-e_{K, \nu}(x^{\nu})\right|=e_{K, \mu}(x^{\mu})- e_{K, \nu}(x^{\nu})\leq e_{K, \mu}(x^{\nu})- e_{K, \nu}(x^{\nu})\leq \left\Vert e_{K, \mu_{\infty}}- e_{K, \mu_{n}} \right\Vert_{\sup}\nonumber. 
\end{align}
If $e_{K, \mu}(x^{\mu})\leq e_{K, \nu}(x^{\nu})$, we have the same inequality by the same reasoning. }

\smallskip
Moreover, 
\begin{align}
\mathcal{D}_{K, \mu_{\infty}}(x^{(n)})&-\inf_{x\in(\mathbb{R}^{d})^{K}}\mathcal{D}_{K, \mu_{\infty}}(x)=\mathcal{D}_{K, \mu_{\infty}}(x^{(n)})-\mathcal{D}_{K, \mu_{\infty}}(x^{(\infty)})\nonumber\\
&\leq \big[e_{K, \mu_{\infty}}(x^{(n)})+e_{K, \mu_{\infty}}(x^{(\infty)})\big]\big(e_{K, \mu_{\infty}}(x^{(n)})-e_{K, \mu_{\infty}}(x^{(\infty)})\big)\nonumber\\
&\leq2\big[e_{K, \mu_{\infty}}(x^{(n)})-e_{K, \mu_{\infty}}(x^{(\infty)})+2e_{K, \mu_{\infty}}(x^{(\infty)})\big]\cdot\mathcal{W}_{2}(\mu_{n}, \mu_{\infty}) \quad\text{\big(by (\ref{perforerror})\big)}\nonumber\\
&\leq4\big[\mathcal{W}_{2}(\mu_{n}, \mu_{\infty})+e_{K, \mu_{\infty}}^{*}\big]\cdot\mathcal{W}_{2}(\mu_{n}, \mu_{\infty})\quad\text{\big(by (\ref{perforerror})\big)}\nonumber\\
&\leq4e_{K, \mu_{\infty}}^{*}\mathcal{W}_{2}(\mu_{n}, \mu_{\infty})+4\mathcal{W}_{2}^{2}(\mu_{n}, \mu_{\infty}).\hfill\qedhere\nonumber
\end{align}
\end{proof}
}

Let $B(x, r)$ denote the ball centered at $x$ with radius $r$. Recall that $F_{K}\coloneqq \big\{ x=(x_{1}, ..., x_{K})\in(\mathbb{R}^{d})^{K}\;\big|\; x_{i}\neq x_{j}, \text{ if } i\neq j\big\}
$. Remark that if $x\in F_{K}$,  then every $y\in B\big(x, \frac{1}{3}\min_{1\leq i,j\leq K, i\neq j} \left|x_{i}-x_{j}\right|\big)$ still lies in $F_{K}$.
In the following theorem, we give an estimate of the convergence rate of the optimal quantizer sequence $x^{(n)}, n\in\mathbb{N}^{*}$.

\begin{theorem}[Convergence rate of optimal quantizers]\label{cvgrate}
Let $K\in\mathbb{N}^{*}$ be the quantization level.  For every $n\in\mathbb{N}^{*}\cup\{\infty\}$, let $\mu_{n}\in\mathcal{P}_{2}(\mathbb{R}^{d})$ with $\mathrm{card}\big(\mathrm{supp}(\mu_{n})\big)\geq K$.  Assume that $\mathcal{W}_{2}(\mu_{n}, \mu_{\infty})\rightarrow0$ as $n\rightarrow+\infty$.  For every $n\in\mathbb{N}^{*}$, let $x^{(n)}$ be an optimal quantizer of $\mu_{n}$ and let $G_{K}(\mu_{\infty})$ denote the set of all optimal quantizers of $\mu_{\infty}$.  If the following assumptions hold
\begin{enumerate}[(a)]
\item  {\color{black}the distortion function $\mathcal{D}_{K, \mu_{\infty}}$ is twice differentiable at every $x\in F_{K}$;}
\item {\color{black}$\mathrm{card}\big(G_{K}(\mu_{\infty})\big)<+\infty$;}
\item for every $x^{(\infty)}\in G_{K}(\mu_{\infty})$, the Hessian matrix of $\mathcal{D}_{K, \mu_{\infty}}$, denoted by $H_{\mathcal{D}_{K, \mu_{\infty}}}$, is positive definite in the neighbourhood of $x^{(\infty)}$ having eigenvalues lower bounded by some $\lambda^{*}>0$, 
\end{enumerate}

%\vspace{-0.4cm}
\noindent then, for $n$ large enough, 
\[d\big(x^{(n)}, G_{K}(\mu_{\infty})\big)^{2}\leq \frac{8}{\lambda^{*}}e_{K, \mu_{\infty}}^{*}\cdot\mathcal{W}_{2}(\mu_{n},\mu_{\infty})+\frac{8}{\lambda^{*}}\cdot\mathcal{W}^{2}_{2}(\mu_{n},\mu_{\infty}).\] 
%where $K^{(1)}=$ and $K^{(2)}=$. %$\limsup_{n}K_{n}^{(1)}\leq \frac{2}{\lambda_{*}}$ and $\limsup_{n}K_{n}^{(2)}\leq\frac{8}{\lambda_{*}}\,e^{*}_{K, \mu_{\infty}}$.
%{How to explain $\lambda_{*}$???????}
\end{theorem}

{\color{black}
\begin{remark}
Section 3 provides a detailed discussion of the conditions in Theorem \ref{cvgrate} and their relation between each other. 

\noindent (1) First,  
in Section \ref{positivedef}, 
we establish {\color{black}that if $\mu_{\infty}$ is $1$-radially controlled, then its distortion function}  $\mathcal{D}_{K, \mu_{\infty}}$ is twice continuously differentiable at every $x\in F_{K}$ and give an exact formula of the Hessian matrix $H_{\mathcal{D}_{K, \mu_{\infty}}}(x)$ in Proposition \ref{contihg}. Thus, one may obtain Condition $(c)$ either by an explicit computation or by numerical methods. Moreover, %Proposition \ref{contihg} implies the continuity of the Hessian matrix  $H_{\mathcal{D}_{K, \mu}}$ so that 
if $H_{\mathcal{D}_{K, \mu}}$ is positive definite at $x\in F_{K}$, it is also positive definite in its neighbourhood.   In Section \ref{logcondim1}, we establish %a sufficient condition for the continuity of every term in the Hessian matrix in dimension $d$ (see  Lemma \ref{contihg}) and 
several sufficient conditions for the positive definiteness of the Hessian matrix $H_{\mathcal{D}_{K, \mu_{\infty}}}$ in the neighbourhood of $x^{(\infty)}\in G_{K}(\mu_{\infty})$ in one dimension. 

\smallskip
\noindent (2) If the distribution $\mu_{\infty}$ is $1$-radially controlled,  a necessary condition for  Condition $(c)$ is Condition $(b)$ %\[\mathrm{card}\big(G_{K}(\mu_{\infty})\big)<+\infty.\] 
(see Lemma \ref{defpo}). Thus, if $\mathrm{card}\big(G_{K}(\mu_{\infty})\big)=+\infty$, it is more reasonable to consider the non-asymtotic upper bound of the performance (Theorem \ref{cvgratedistortion}) to study the convergence rate of the optimal quantization. A typical example is the standard multidimensional normal distribution $\mu_{\infty}=\mathcal{N}(0, I_{d})$: it is $1$-radially controlled and any rotation of an optimal quantizer $x$ is still  optimal  so that  $\mathrm{card}\big(G_{K}(\mu_{\infty})\big)=+\infty$. 
\end{remark}}

\smallskip
\begin{proof}[Proof of Theorem \ref{cvgrate}] Since the quantization level $K$ is fixed throughout the proof, we will drop  the subscripts $K$ and $\mu$ of the distortion function $\mathcal{D}_{K,\,\mu}$ and we will denote by $\mathcal{D}_{n}$ (\textit{respectively,} $\mathcal{D}_{\infty})$ the distortion function of $\mu_{n}$ (\textit{resp.}  $\mu_{\infty}$). 

After Pollard's theorem, $(x^{(n)})_{n\in\mathbb{N}^{*}}$ is bounded and any limiting point of $x^{(n)}$ lies in $G_{K}(\mu_{\infty})$. We may assume that, up to {\color{black}the extraction of a subsequence of} $x^{(n)}$, still denoted by $x^{(n)}$, we have $x^{(n)}\rightarrow x^{(\infty)}\in G_{K}(\mu_{\infty})$. Hence  $d\big(x^{(n)}, G_{K}(\mu_{\infty})\big)\leq \left|x^{(n)}-x^{(\infty)}\right|$.
%It follows from (\ref{gradgop}) that $\mathcal{D}_{\infty}$ is differentiable at $x^{(\infty)}$. 

%{By Lemma \ref{continuhessian} and Proposition \ref{relation}-(ii),  Condition $(a)$ implies that 
{\color{black}Proposition \ref{relation} implies that $x^{(\infty)}\in F_{K}$. As $\mathcal{D}_{\infty}$ is twice differentiable at $x^{(\infty)}$,} 
%Hence, 
the second order Taylor expansion of $\mathcal{D}_{\infty}$ at $x^{(\infty)}$ reads: 
\[\mathcal{D}_{\infty}(x^{(n)})=\mathcal{D}_{\infty}(x^{(\infty)})+\big\langle\nabla\mathcal{D}_{\infty}(x^{(\infty)})\mid x^{(n)}-x^{(\infty)}\big\rangle+\frac{1}{2} H_{\mathcal{D}_{\infty}}(\zeta^{(n)})(x^{(n)}-x^{(\infty)})^{\otimes2},\]
where $H_{\mathcal{D}_{\infty}}$ denotes the Hessian matrix of $\mathcal{D}_{\infty}$, $\zeta^{(n)}$ lies in the geometric segment $(x^{(n)}, x^{(\infty)})$ and for a matrix $A$ and a {vector} $u$, $Au^{\otimes2}$ stands for $u^{T}Au$. 

As $x^{(\infty)}\in G_{K}(\mu_{\infty})=\text{argmin}\,\mathcal{D}_{\infty}$ and $\text{card}\big(\text{supp}(\mu_{\infty})\big)\geq K$, one has $\nabla\mathcal{D}_{\infty}(x^{(\infty)})=0$. %by applying Fermat's theorem on stationary point.  
Hence
\begin{equation}\label{errrr}
\mathcal{D}_{\infty}(x^{(n)})-\mathcal{D}_{\infty}(x^{(\infty)})=\frac{1}{2}H_{\mathcal{D}_{\infty}}(\zeta^{(n)})(x^{(n)}-x^{(\infty)})^{\otimes2}.
\end{equation}
It follows from Theorem \ref{cvgratedistortion} that 
\begin{align}\label{chan}
H_{\mathcal{D}_{\infty}}(\zeta^{(n)})(x^{(n)}-x^{(\infty)})^{\otimes 2}&=2\big(\mathcal{D}_{\infty}(x^{(n)})-\mathcal{D}_{\infty}(x^{(\infty)})\big)\nonumber\\
&\leq 8e_{K, \mu_{\infty}}^{*}\mathcal{W}_{2}(\mu_{n}, \mu_{\infty})+8\mathcal{W}_{2}^{2}(\mu_{n}, \mu_{\infty}).
\end{align}

By {\color{black}Condition (c)}, $H_{\mathcal{D}_{\infty}}$ is assumed to be positive definite in the neighbourhood of all $x^{(\infty)}\in G_{K}(\mu_{\infty})$ having eigenvalues lower bounded by some $\lambda^{*}>0$. As $\zeta^{(n)}$ lies in the geometric segment $(x^{(n)}, x^{(\infty)})$ and $x^{(n)}\rightarrow x^{(\infty)}$,  
there exists an $n_{0}(x^{(\infty)})$ such that for all $n\geq n_{0}$, $H_{\mathcal{D}_{\infty}}(\zeta^{(n)})$ is a positive definite matrix. It follows that, for $n\geq n_{0}$,
\begin{align}
\lambda^{*}\left| x^{(n)}-x^{(\infty)}\right|^{2}&\leq H_{\mathcal{D}_{\infty}}(\zeta^{(n)})(x^{(n)}-x^{(\infty)})^{\otimes 2}\nonumber\\
&\leq 8e_{K, \mu_{\infty}}^{*}\mathcal{W}_{2}(\mu_{n}, \mu_{\infty})+8\mathcal{W}_{2}^{2}(\mu_{n}, \mu_{\infty}).\nonumber
\end{align}
Thus, one can directly conclude by multiplying  at each side of the above inequality by $\frac{1}{\lambda^{*}}$. 
\end{proof}

{ Based on conditions in Theorem \ref{cvgrate}, if  we know the exact limit of the optimal quantizer sequence $x^{(n)}$, % of $\mu_{n}$ converges to some point $x^{(\infty)}$, 
we have the following result whose proof is similar to that of Theorem \ref{cvgrate}.} % is established for the situation that we did not know

\begin{corollary}\label{corrate}
Let $K\in\mathbb{N}^{*}$ be the quantization level.  For every $n\in\mathbb{N}^{*}\cup\{\infty\}$, let $\mu_{n}\in\mathcal{P}_{2}(\mathbb{R}^{d})$ with $\mathrm{card}\big(\mathrm{supp}(\mu_{n})\big)\geq K$.  Assume that $\mathcal{W}_{2}(\mu_{n}, \mu_{\infty})\rightarrow0$ as $n\rightarrow+\infty$. Let $x^{(n)}\in\mathrm{argmin}\;\mathcal{D}_{K,\,\mu_{n}}$ such that $\lim_{n} x^{(n)}\rightarrow x^{(\infty)}$. If the Hessian matrix of $\mathcal{D}_{K,\,\mu_{\infty}}$ is  positive definite  in the neighbourhood of $x^{(\infty)}$, then, for $n$ large enough,
 \[\left| x^{(n)}-x^{(\infty)}\right|^{2} \leq C^{(1)}_{\mu_{\infty}}\cdot\mathcal{W}_{2}(\mu_{n}, \mu_{\infty})+C^{(2)}_{\mu_{\infty}}\cdot\mathcal{W}_{2}^{2}(\mu_{n}, \mu_{\infty}),\]
where $C^{(1)}_{\mu_{\infty}}$ and $C^{(2)}_{\mu_{\infty}}$ are real constants only depending on $\mu_{\infty}$. 
\end{corollary}

\section{Hessian matrix $H_{\mathcal{D}_{K,\,\mu}}$ of {the} distortion function $\mathcal{D}_{K,\,\mu}$}\label{positivedef}

\noindent Let $\mu\in\mathcal{P}_{2}(\mathbb{R}^{d})$ with $\mathrm{card}\big(\mathrm{supp}(\mu)\big)\geq K$ and let $x^{*}$ be an optimal quantizer of $\mu$ at level $K$. In Section \ref{continuityhg}, we show conditions under which  the distortion function $\mathcal{D}_{K, \mu}$ is twice differentiable and give %Lemma \ref{continuhessian} by giving 
the exact formula {of its} Hessian matrix $H_{\mathcal{D}_{K,\,\mu}}$. % of the distortion function $\mathcal{D}_{K,\,\mu}$ when $\mu$ is absolutely continuous with  respect to the Lebesgue measure $\lambda_{d}$ on $\mathbb{R}^{d}$, {with a continuous density function $f$. 
%Moreover, we also give a criterion for the continuity of every term of the Hessian matrix $H_{\mathcal{D}_{K,\,\mu}}$ and a necessary condition for the  positive definitiveness of the Hessian matrix $H_{\mathcal{D}_{K,\,\mu}}(x^{*})$.
% written by $\mu(d\xi)=f(\xi)\lambda_{d}(d\xi)$, where $f$ denotes the density function of $\mu$. 
In Section \ref{logcondim1}, we give {several criterions} for the positive definiteness of the Hessian matrix $H_{\mathcal{D}_{K,\,\mu}}$ in the neighbourhood of an optimal quantizer $x^{*}$ in dimension 1. %{Finally, in Section \ref{posdefdimd},  }

\subsection{Hessian matrix $H_{\mathcal{D}_{K,\,\mu}}$ on $\mathbb{R}^{d}$}\label{continuityhg}

\noindent If $\mu$ is absolutely continuous with  respect to  the Lebesgue measure $\lambda_{d}$ on $\mathbb{R}^{d}$ with the density function $f$, then the distortion function $\mathcal{D}_{K,\,\mu}$ is differentiable (see \cite{pages1998space})  at all point $x=(x_{1}, ..., x_{K})\in F_{K}$ with% when $x_{i}\neq x_{j}, i\neq j$
\begin{equation}\label{gradg}
\frac{\partial\mathcal{D}_{K,\,\mu}}{\partial x_{i}} (x)=2\int_{V_{i}(x)}(x_{i}-\xi)f(\xi)\lambda_{d}(d\xi), \;\;\;\text{for}\;i=1,...,K.
\end{equation}

\noindent In the following Proposition, we give a criterion for the twice differentiability of the distortion function $\mathcal{D}_{K, \mu}$. 

\begin{proposition}\label{contihg}
Let $\mu\in\mathcal{P}_{2}(\mathbb{R}^{d})$ be absolutely continuous with  respect to the Lebesgue measure $\lambda_{d}$ on $\mathbb{R}^{d}$ {with a continuous density function $f$}. If $\mu$ is $1$-radially controlled, then 
\begin{enumerate}[$(i)$]
\item the distortion function $\mathcal{D}_{K,\,\mu}$ is twice differentiable at every $x\in F_{K}$ and the Hessian matrix $H_{\mathcal{D}_{K, \mu}}(x)=\Big[\frac{\partial^{2}\mathcal{D}_{K,\,\mu}}{\partial x_{j}\partial x_{i}}(x)\Big]_{1\leq i\leq j\leq K}$ is defined by
\small
\begin{flalign}\label{deriveeseconde1}
&\frac{\partial^{2}\mathcal{D}_{K,\,\mu}}{\partial x_{j}\partial x_{i}}(x)=2\int_{V_{i}(x)\cap V_{j}(x)}(x_{i}-\xi)\otimes(x_{j}-\xi)\cdot  \frac{1}{\left| x_{j}-x_{i}\right|}f(\xi)\lambda_{x}^{ij}(d\xi), \;\;\;\text{if $j\neq i$,}&\\
\label{deriveeseconde2}
&\frac{\partial^{2}\mathcal{D}_{K,\,\mu}}{\partial x_{i}^{2}}(x)=2\Big[\mu\big(V_{i}(x)\big)\mathrm{I}_{d} - \sum_{\substack{i\neq j \\ 1\leq j\leq K}}\int_{V_{i}(x)\cap V_{j}(x)}(x_{i}-\xi)\otimes(x_{i}-\xi)\cdot  \frac{1}{\left| x_{j}-x_{i}\right|}f(\xi)\lambda_{x}^{ij}(d\xi)\Big], &%i=1,..., K
\end{flalign}
\normalsize
where in (\ref{deriveeseconde1}) and (\ref{deriveeseconde2}), $u\otimes v\coloneqq[u^{i}v^{j}]_{1\leq i, j \leq d}$  for any two vectors $u=(u^{1}, ..., u^{d})$ and $v=(v^{1}, ... , v^{d})$ in $\mathbb{R}^{d}$;
\item  every element $\frac{\partial^{2}\mathcal{D}_{K,\,\mu}}{\partial x_{j}\partial x_{i}}$ of the Hessian matrix  $H_{\mathcal{D}_{K, \mu}}$  is  continuous at every $x\in F_{K}$. 
\end{enumerate}
\end{proposition}

The proof of Proposition \ref{contihg} is postponed to Appendix C. The following lemma shows that under the condition of Proposition \ref{contihg}, %  if the Hessian matrix  $H_{\mathcal{D}_{K, \mu}}$ is positive definite at $x^{*}$, it is also positive definite in its neighbourhood. Moreover, 
Condition (c) in  Theorem \ref{cvgrate} implies Condition (b).%$\mathrm{card}\big(G_{K}(\mu_{\infty})\big)<+\infty$. 
\begin{lemma}\label{defpo}
Let $\mu\in\mathcal{P}_{2}(\mathbb{R}^{d})$ be absolutely continuous with the respect to the Lebesgue measure $\lambda_{d}$ on $\mathbb{R}^{d}$ {with a continuous density function $f$}. If {$\mu_{\infty}$ is $1$-radially controlled} and 
$\mathrm{card}\big(G_{K}(\mu_{\infty})\big)=+\infty$, then there exists a point $x\in G_{K}(\mu_{\infty})$  such that the Hessian matrix  $H_{\mathcal{D}_{K, \mu_{\infty}}}$ of $\mathcal{D}_{K, \mu_{\infty}}$  at $x$ has an eigenvalue $0$. 
\end{lemma}
%\begin{remark} If $\mu_{\infty}$ satisfies the condtions in Lemma \ref{defpo} and if $\mathrm{card}\big(G_{K}(\mu_{\infty})\big)<+\infty$, a sufficient condition of Condition (c) in Theorem \ref{cvgrate} is that $H_{\mathcal{D}_{K, \mu_{\infty}}}$ is positive definite at every $x\in G_{K}(\mu_{\infty})$. In this case, one can take $\lambda^{*}=\min_{x\in G_{K}(\mu_{\infty})}\underline{\lambda}_{H_{\mathcal{D}_{K, \mu_{\infty}}}(x)}-\varepsilon$ for a $\varepsilon>0$, where $\underline{\lambda}_{A}$ denotes the smallest eigenvalue of a matrix $A$. 
%\end{remark}
\begin{proof}[Proof of Lemma \ref{defpo}]
We denote by $H_{\mathcal{D}_{\infty}}$ instead of  $H_{\mathcal{D}_{K, \mu_{\infty}}}$ to simplify the notation. 
Proposition \ref{relation} implies that $G_{K}(\mu_{\infty})$ is a compact set. If $\mathrm{card}\big(G_{K}(\mu_{\infty})\big)=+\infty$, there exists $x, x^{(k)}\in G_{K}(\mu_{\infty}), k\in\mathbb{N}^{*}$ such that $x^{(k)}\rightarrow x$ when $k\rightarrow+\infty$. Set $u_{k}\coloneqq \frac{x^{(k)}-x}{\left|x^{(k)}-x\right|}$, $k\geq1$, then we have $\left| u_{k}\right|=1$ for all $k\in\mathbb{N}^{*}$. Hence, there exists a subsequence $\varphi(k)$ of $k$ such that $u_{\varphi(k)}$ converges to some $\widetilde{u}$ with $\left|\widetilde{u}\right|=1$. 

The Taylor expansion of $\mathcal{D}_{K, \mu_{\infty}}$ at $x$ reads:
\[\mathcal{D}_{K, \mu_{\infty}}(x^{\varphi(k)})=\mathcal{D}_{K, \mu_{\infty}}(x)+\big\langle\nabla \mathcal{D}_{K, \mu_{\infty}}(x)\;\big{|}\; x^{\varphi(k)}-x\big\rangle + \frac{1}{2} H_{\mathcal{D}_{\infty}}(\zeta^{\varphi(k)})(x^{\varphi(k)}-x)^{\otimes 2},\]
where $\zeta^{\varphi(k)}$ lies in the geometric segment $(x^{\varphi(k)}, x)$. Since $x, x^{(k)}, k\in\mathbb{N}^{*}\in G_{K}(\mu_{\infty})$, then $\nabla \mathcal{D}_{K, \mu_{\infty}}(x)=0$ and for any $k\in\mathbb{N}^{*}$, $\mathcal{D}_{K, \mu_{\infty}}(x^{\varphi(k)})=\mathcal{D}_{K, \mu_{\infty}}(x)$. Hence, for any $k\in\mathbb{N}^{*}$, $H_{\mathcal{D}_{\infty}}(\zeta^{\varphi(k)})(x^{\varphi(k)}-x)^{\otimes 2}=0$. Consequently, for any $k\in\mathbb{N}^{*}$, 
\[H_{\mathcal{D}_{\infty}}(\zeta^{\varphi(k)})\Big(\frac{x^{\varphi(k)}-x}{\left|x^{\varphi(k)}-x\right|}\Big)^{\otimes 2}=0.\]
Thus we have $H_{\mathcal{D}_{\infty}}(x)\widetilde{u}^{\otimes 2}=0$ by letting $k\rightarrow +\infty$, so that $H_{\mathcal{D}_{\infty}}(x)$ has an eigenvalue $0$. 
\end{proof}

\subsection{A criterion for positive definiteness of $H_{\mathcal{D}_{\infty}}(x^{*})$ in 1-dimension}\label{logcondim1}
\noindent Let  $\mu\in\mathcal{P}_{2}(\mathbb{R})$ with $\mathrm{card}\big(\mathrm{supp}(\mu)\big)\geq K$. Assume that $\mu$ is absolutely continuous with  respect to  the Lebesgue measure {having a density function $f$}. %, written by $\mu(d\xi)=f(\xi)d\xi$.  
 In the one-dimensional case, it is useful to point out a sufficient condition for the uniqueness of optimal quantizer. A probability distribution $\mu$ is called \textit{strongly unimodal} if its density function $f$ satisfies that $I=\{f>0\}$ is an open (possibly unbounded) interval and $\log f$ is concave on $I$. Let $F_{K}^{+}:=\big\{ x=(x_{1}, ..., x_{K})\in \mathbb{R}^{K}\mid -\infty< x_{1} < x_{2} < ... <x_{K} <+\infty \big\}$. %We have the uniqueness of optimal quantizer for such distributions. 
\begin{lemma}\label{uniqueness}
For $K\in\mathbb{N}^{*}$, if $\mu$ is strongly unimodal with $\mathrm{card}\big(\mathrm{supp}(\mu)\big)\geq K$, then there is only one stationary (then optimal) quantizer of level $K$ in $F_{K}^{+}$. 
\end{lemma}

We refer to \cite{kieffer1983uniqueness}, \cite{trushkin1982sufficient}, \cite{MR1239840} and \cite{graf2000foundations}[see Theorem 5.1] for the proof of Lemma \ref{uniqueness} and for more details. 

Given a $K$-tuple $x=(x_{1}, ..., x_{K})\in F_{K}^{+}$, the Voronoi region $V_{i}(x)$ can be explicitly written: $V_{1}(x)=(-\infty, \frac{x_{1}+x_{2}}{2}]$, $V_{K}(x)=[\frac{x_{K-1}+x_{K}}{2}, +\infty)$ and $V_{i}(x)=[\frac{x_{i-1}+x_{i}}{2}, \frac{x_{i}+x_{i+1}}{2}]$ for $i=2,...,K-1$. 
For all $x\in F_{K}^{+}$, $\mathcal{D}_{K,\,\mu}$ is differentiable at $x$ and by (\ref{gradg}) and
 \begin{equation}\label{Gdim1}
\nabla\mathcal{D}_{K,\,\mu}(x)=\left[\int_{V_{i}(x)}2(x_{i}-\xi)f(\xi)d\xi\right]_{i=1, ..., K}.
\end{equation}
Therefore, as $\nabla\mathcal{D}_{K,\,\mu}(x^{*})=0$, one can solve the optimal quantizer $x^{*}\in F_{K}^{+}$ as follows,   
\begin{equation}\label{stationarych3}
x_{i}^{*}=\frac{\int_{V_{i}(x^{*})}\xi f(\xi)d\xi}{\mu\big(V_{i}(x^{*})\big)},\;\;\;\text{for}\; i=1, ..., K.
\end{equation}

For any $x\in F_{K}^{+}$, the Hessian matrix $H_{\mathcal{D}_{K, \mu}}$ of $\mathcal{D}_{K,\,\mu}$ at $x$ is a tridiagonal symmetry matrix and can be calculated as follows, 
\small{
\begin{equation}\label{hessiendim11}
H_{\mathcal{D}_{K, \mu}}(x)=\left(\begin{array}{ccccccc}
A_{1}-B_{1,2} & -B_{1,2}\\
& \ddots & \\
   & -B_{i-1,i} & A_{i}-B_{i-1, i}-B_{i, i+1} & -B_{i, i+1}\\
   &  &  & \ddots & \\
   &  &  &  & -B_{K-1, K} & A_{K}-B_{K-1,K}
\end{array}\right),
\end{equation}}
\normalsize where $A_{i}=2\mu\big(C_{i}(x)\big)$ for $1\leq i \leq K$ and $B_{i, j}=\frac{1}{2}(x_{j}-x_{i})f(\frac{x_{i}+x_{j}}{2})$ for $1\leq i < j\leq K$. 
{ Let $F_{\mu}$ be the cumulative distribution function of $\mu$, then 
\begin{align}
&A_{1}=2\mu\big(C_{1}(x)\big)=2F_{\mu}\Big(\frac{x_{1}+x_{2}}{2}\Big),\nonumber\\
&A_{i}=2\mu\big(C_{i}(x)\big)=2\Big[F_{\mu}\Big(\frac{x_{i+1}+x_{i}}{2}\Big)-F_{\mu}\Big(\frac{x_{i-1}+x_{i}}{2}\Big)\Big], \quad\text{for }\, i=2, ..., K-1,\nonumber\\
&A_{K}=2\mu\big(C_{K}(x)\big)=2\Big[1-F_{\mu}\Big(\frac{x_{K-1}+x_{K}}{2}\Big)\Big].\nonumber
\end{align} 
Then the continuity of each term in the matrix $H_{\mathcal{D}_{K, \mu}}(x)$ can be directly derived from the continuity of $F_{\mu}$.% and $f$.}

For $1\leq i\leq K$, we define $\displaystyle L_{i}(x)\coloneqq\sum_{j=1}^{K}\frac{\partial^{2}\mathcal{D}_{K,\,\mu}}{\partial x_{i}\partial x_{j}}(x)$. The following proposition gives sufficient conditions to obtain the positive definiteness of $H_{\mathcal{D}_{K, \mu}}(x^{*})$. 

\begin{proposition}\label{podedim1} Let  $\mu\in\mathcal{P}_{2}(\mathbb{R})$ with $\mathrm{card}\big(\mathrm{supp}(\mu)\big)\geq K$. Assume that $\mu$ is absolutely continuous with  respect to  the Lebesgue measure {having a density function $f$}.
Any of the following two conditions implies the positive definiteness of $H_{\mathcal{D}_{K, \mu}}(x^{*})$,
\begin{enumerate}[(i)]
\item $\mu$ is the uniform distribution,
\item $f$ is differentiable and $\log f$ is strictly concave. 
\end{enumerate}
In particular, $(ii)$ also implies that $L_{i}(x^{*})>0$, $i=1, ..., K$. 

\end{proposition}
Proposition \ref{podedim1} is proved in Appendix D.
Remark that, under the conditions of Proposition \ref{podedim1}, $\mu$ is strongly unimodal so that there is exactly one optimal quantizer in $F_{K}^{+}$ for $\mu$ at level $K$.
%, if $x^{*}=(x_{1}^{*}, ..., x_{K}^{*})\in F_{K}^{+}\cap\mathrm{argmin}\,\mathcal{D}_{K,\,\mu}$, then $\Gamma^{*}=\{x_{1}, ..., x_{K}\}$ is the unique optimal quantizer for $\mu$ at level $K$ (viewed as a set).  
The conditions in Proposition \ref{podedim1} directly imply the following convergence rate results.

\begin{theorem}\label{covratedim1}
{\color{black}Let $K\in\mathbb{N}^{*}$ be the quantization level. For every $n\in\mathbb{N}^{*}\cup \{\infty\}$, let $\mu_{n}\in\mathcal{P}_{2}(\mathbb{R})$ with $\mathrm{card}\big(\mathrm{supp}(\mu_{n})\big)\geq K$ be such that $\mathcal{W}_{2}(\mu_{n}, \mu_{\infty})\rightarrow0$ as $n\rightarrow+\infty$. Assume that $\mu_{\infty}$ is absolutely continuous with  respect to  the Lebesgue measure, written $\mu_{\infty}(d\xi)=f(\xi)d\xi$. Let $x^{(n)}$ be an optimal quantizer of $\mu_{n}$ converging to $x^{(\infty)}$. Then any one of the following two conditions
\begin{enumerate}[$(i)$]
\item $\mu_{\infty}$ is the uniform distribution
\item $f$ is differentiable and $\log f$ is strictly concave
\end{enumerate}
implies the existence of constants $C_{\mu_{\infty}}^{(1)}$ and $C_{\mu_{\infty}}^{(2)}$ only depending  on $\mu_{\infty}$ such that for $n$ large enough,
\[\left|x^{(n)}-x^{(\infty)}\right|^{2}\leq C_{\mu_{\infty}}^{(1)}\cdot\mathcal{W}_{2}(\mu_{n}, \mu_{\infty})+C_{\mu_{\infty}}^{(2)}\cdot\mathcal{W}_{2}^{2}(\mu_{n}, \mu_{\infty}).\]

}
\end{theorem}
\begin{proof}
Let $\mathcal{D}_{K,\,\mu_{\infty}}$ denote the distortion function of $\mu_{\infty}$ and let $H_{\mathcal{D}_{\infty}}$ denote the Hessian matrix of $\mathcal{D}_{K,\,\mu_{\infty}}$.

\noindent$(i)$ Let $g_{k}(x)$ be the $k$-th leading principal minor of $H_{\mathcal{D}_{\infty}}(x)$ defined in ($\ref{hessiendim11}$), then  $g_{k}(x), k=1, ..., K,$ are continuous functions in $x$ since every element in this matrix is continuous. Proposition \ref{podedim1} implies $g_{k}(x^{(\infty)})>0$, thus there exists $r>0$ such that for every $x\in B(x^{(\infty)}, r)$, $g_{k}(x^{(\infty)})>0$ so that $H_{\mathcal{D}_{\infty}}(x)$ is positive definite. What remains can be directly proved by Corollary \ref{corrate}.

\noindent$(ii)$ The function $\displaystyle L_{i}(x)\coloneqq\sum_{j=1}^{K}\frac{\partial^{2}\mathcal{D}_{K,\,\mu_{\infty}}}{\partial x_{i}\partial x_{j}}(x)$ is continuous on $x$ and Proposition \ref{podedim1} implies that $L_{i}(x^{(\infty)})>0$. Hence, there exists $r>0$ such that $\forall x\in B(x^{(\infty)}, r)$, $L_{i}(x)>0$. From (\ref{hessiendim11}), one can remark that the $i$-th diagonal elements in $H_{\mathcal{D}_{\infty}}(x)$ is always larger than $L_{i}(x)$ for any $x\in\mathbb{R}^{K}$, then after Gershgorin Circle theorem, we derive that $H_{\mathcal{D}_{\infty}}(x)$ is positive definite for every $x\in B(x^{(\infty)}, r)$. What remains can be directly proved by Corollary \ref{corrate}.
\end{proof}

%%%%%%%%%%%%%%%%%%
%%%%%%%%%%%%%%%%%%

%%%%%%%%%%%%%
%%%%%%%%%%%%%
%%%%%%%%%%%%%

\section{Empirical measure case}\label{section3}

\noindent  Let $K\in\mathbb{N}^{*}$ be the quantization level. 
 Let $\mu\in\mathcal{P}_{2+\varepsilon}(\mathbb{R}^{d})$ for some $\varepsilon>0$ and $\mathrm{card}\big(\mathrm{supp}(\mu)\big)\geq K$. Let $X$ be a random variable with distribution $\mu$ and let $(X_{n})_{n\geq1}$ be a sequence of independent identically distributed $\mathbb{R}^{d}$-valued random variables with probability distribution $\mu$. The empirical measure is defined for every $n\in\mathbb{N}^{*}$ by
\begin{equation}\label{empi}
\mu_{n}^{\omega}\coloneqq \frac{1}{n}\sum_{i=1}^{n}\delta_{X_{i}(\omega)},\;\;\;\omega\in\Omega,
\end{equation}
where $\delta_{a}$ is the Dirac mass at $a$. For $n\geq1$, let $x^{(n),\omega}$ be an optimal quantizer of $\mu_{n}^{\omega}$. The superscript $\omega$ is to emphasize that both $\mu_{n}^{\omega}$ and $x^{(n), \omega}$ are random and we will drop $\omega$ when there is no ambiguity. 
We  cite two results of the convergence of $\mathcal{W}_{2}(\mu_{n}^{\omega}, \mu)$ among so many researches in this topic: the a.s. convergence in  \cite{pollard1982quantization}[see Theorem 7]  and the $L^{p}$-convergence rate of $\mathcal{W}_{p}(\mu_{n}^{\omega}, \mu)$  in \cite{fournier2015rate}.

\begin{thm*}(\cite{fournier2015rate}[see Theorem 1])
Let $p>0$ and let $\mu\in\mathcal{P}_{q}(\mathbb{R}^{d})$ for some $q>p$. Let $\mu_{n}^{\omega}$ denote the empirical measure of $\mu$ defined in (\ref{empi}). There exists a constant $C$ only depending on $p,d,q$ such that, for all $n\geq 1$,
\begin{equation}\label{nico}
\small{\mathbb{E}\Big(\mathcal{W}_{p}^{p}(\mu_{n}^{\omega}, \mu)\Big)\leq CM_{q}^{p/q}(\mu)\times\begin{cases}
n^{-1/2}+n^{-(q-p)/q} & \:\mathrm{if}\:p>d/2\:\mathrm{and}\:q\neq2p\\
n^{-1/2}\log(1+n)+n^{-(q-p)/q} & \:\mathrm{if}\:p=d/2\:\mathrm{and}\:q\neq2p\\
n^{-p/d}+n^{-(q-p)/q} & \:\mathrm{if}\:p\in(0, d/2)\:\mathrm{and}\:q\neq d/(d-p)
\end{cases}},
\end{equation}
where $M_{q}(\mu)=\int_{\mathbb{R}^{d}}\left|\xi\right|^{q}\mu(d\xi)$.
\end{thm*}
\normalsize

Let $\mathcal{D}_{K,\,\mu}$ denote the distortion function of $\mu$ and let $\mathcal{D}_{K,\,\mu_{n}}$ denote the distortion fuction of $\mu_{n}^{\omega}$  for any $n\in\mathbb{N}^{*}$. Recall by Definition \ref{defdistortion} that for $c=(c_{1}, ..., c_{K})\in(\mathbb{R}^{d})^{K}$,
\begin{align}
&\mathcal{D}_{K,\,\mu}(c)=\mathbb{E}\min_{1\leq k\leq K}\left| X-c_{k}\right|^{2}=\mathbb{E}\Big[\left|X\right|^{2}+\min_{1\leq k\leq K}\big(-2\langle X | c_{k}\rangle+\left|c_{k}\right|^{2}\big)\Big],\nonumber\\
&\text{and}\;\;\mathcal{D}_{K,\,\mu_{n}}(c)=\frac{1}{n}\sum_{i=1}^{n}\min_{1\leq k\leq K}\left| X_{i}-c_{k}\right|^{2}=\frac{1}{n}\sum_{i=1}^{n}\left|X_{i}\right|^{2}+\min_{1\leq k\leq K}\left(-\frac{2}{n}\sum_{i=1}^{n}\langle X_{i} | c_{k}\rangle+\left|c_{k}\right|^{2}\right).\nonumber
\end{align}

The a.s. convergence of optimal quantizers for the empirical measure has been proved in \cite{pollard1981strong}. 
We give a first upper bound of the  clustering performance %of optimal quantizers for the empirical measure 
by applying directly Theorem \ref{cvgratedistortion} and (\ref{nico}).

\begin{proposition}\label{convrateemprical}
Let $K\in\mathbb{N}^{*}$ be the quantization level. 
Let $\mu\in\mathcal{P}_{q}(\mathbb{R}^{d})$ for some $q>2$ with $\text{card}(\text{supp}(\mu))\geq K$ and let $\mu_{n}^{\omega}$ be the empirical measure of $\mu$ defined in (\ref{empi}).   Let $x^{(n),\omega}$ be an optimal quantizer at level $K$ of $\mu_{n}^{\omega}$. Then for any $n> K$, 
{
\begin{align}\label{c}
\mathbb{E}\,\mathcal{D}_{K,\,\mu}&(x^{(n),\omega})-\inf_{x\in(\mathbb{R}^{d})^{K}}\mathcal{D}_{K,\,\mu}(x)\nonumber\\
&\leq C_{d, q,\mu, K}\times \begin{cases}
n^{-1/4}+n^{-(q-2)/2q}& \mathrm{if}\:d<4\:\mathrm{and}\:q\neq4\\
n^{-1/4}\big(\log(1+n)\big)^{1/2}+n^{-(q-2)/2q}  & \mathrm{if}\:d=4\:\mathrm{and}\:q\neq4\\
n^{-1/d}+n^{-(q-2)/2q} & \mathrm{if}\:d>4\:\mathrm{and}\:q\neq d/(d-2)
\end{cases}.
\end{align}
}
\normalsize
where % $q\in(2, 2+\varepsilon)$ and 
$C_{d, q,\mu, K}$ is a constant depending on $d, q, \mu$ and the quantization level $K$. 
\end{proposition}

The reason why we only consider $n> K$ is that for a fixed $n\in\mathbb{N}^{*}$, the empirical measure $\mu_{n}$ defined in (\ref{empi}) is supported by $n$ points, which implies that, if $n\leq K$, the optimal quantizer of $\mu_{n}$ at level $K$, viewed as a set, is in fact $\mathrm{supp}(\mu_{n})$. This makes the above bound of no interest. 
Following the remark after Theorem 1 in \cite{fournier2015rate}, one can remark that if the probability distribution $\mu$ has sufficiently large moments (namely if $q>4$
when $d\leq4$ and $q>2d/(d-2)$ when $d>4$), then the term $n^{-(q-2)/2q}$ is negligible and can be removed. 

{
\begin{proof}[Proof of Proposition \ref{convrateemprical}]
For every $\omega\in\Omega$ and for every $n> K$, Theorem \ref{cvgratedistortion} implies that 
\[\mathcal{D}_{K, \mu}(x^{(n), \omega})-\inf_{x\in(\mathbb{R}^{d})^{K}}\mathcal{D}_{K, \mu}(x)\leq 4e_{K, \mu}^{*}\mathcal{W}_{2}(\mu_{n}^{\omega}, \mu)+4\mathcal{W}_{2}^{2}(\mu_{n}^{\omega}, \mu).\]
Thus, 
\[\mathbb{E}\,\mathcal{D}_{K, \mu}(x^{(n), \omega})-\inf_{x\in(\mathbb{R}^{d})^{K}}\mathcal{D}_{K, \mu}(x)\leq 4e_{K, \mu}^{*}\mathbb{E}\,\mathcal{W}_{2}(\mu_{n}^{\omega}, \mu)+4\,\mathbb{E}\,\mathcal{W}_{2}^{2}(\mu_{n}^{\omega}, \mu).\]
It follows from (\ref{nico}) applied with $p=2$ that 
\begin{align}\label{w2carre}
\mathbb{E}\,\mathcal{W}_{2}^{2}(\mu_{n}^{\omega}, \mu)&\leq C_{d,q,\mu}\times\begin{cases}
n^{-1/2}+n^{-(q-2)/q} & \mathrm{if}\:d<4\:\mathrm{and}\:q\neq4\\
n^{-1/2}\log(1+n)+n^{-(q-2)/q} & \mathrm{if}\:d=4\:\mathrm{and}\:q\neq4\\
n^{-2/d}+n^{-(q-2)/q} & \mathrm{if}\:d>4\:\mathrm{and}\:q\neq d/(d-2)
\end{cases},
\end{align}
where $C_{d,q,\mu}=C\cdot M_{q}^{2/q}(\mu)$ and $C$ is the constant in $(\ref{nico})$.
Moreover, as $\mathbb{E}\,\mathcal{W}_{2}(\mu_{n}^{\omega}, \mu)\leq\big(\mathbb{E}\mathcal{W}_{2}^{2}(\mu_{n}^{\omega}, \mu)\big)^{1/2}$ and $\sqrt{a+b}\leq \sqrt{a}+\sqrt{b} $ for any $a,b\in\mathbb{R}_{+}$, Inequality (\ref{nico}) also implies 
\begin{align}
\mathbb{E}\,\mathcal{W}_{2}(\mu_{n}^{\omega}, \mu)&\leq C_{d,q,\mu}^{1/2}\times\begin{cases}
n^{-1/4}+n^{-(q-2)/2q} & \mathrm{if}\:d<4\:\mathrm{and}\:q\neq4\\
n^{-1/4}\big(\log(1+n)\big)^{1/2}+n^{-(q-2)/2q} & \mathrm{if}\:d=4\:\mathrm{and}\:q\neq4\\
n^{-1/d}+n^{-(q-2)/2q} & \mathrm{if}\:d>4\:\mathrm{and}\:q\neq d/(d-2)
\end{cases}.\nonumber
\end{align}
Consequently, 
\begin{align}\label{empriboundK}
\mathbb{E}\,\mathcal{D}_{K, \mu}(x^{(n), \omega})&-\inf_{x\in(\mathbb{R}^{d})^{K}}\mathcal{D}_{K, \mu}(x)\leq 4e_{K, \mu}^{*}\mathbb{E}\,\mathcal{W}_{2}(\mu_{n}^{\omega}, \mu)+4\,\mathbb{E}\,\mathcal{W}_{2}^{2}(\mu_{n}^{\omega}, \mu).\nonumber\\
&\leq8(C_{d,q,\mu}^{1/2}e_{K, \mu}^{*}\vee C_{d, q, \mu})\times\nonumber\\
&\;\;\;\;\begin{cases}
n^{-1/4}+n^{-(q-2)/2q}& \mathrm{if}\:d<4\:\mathrm{and}\:q\neq4\\
n^{-1/4}\big(\log(1+n)\big)^{1/2}+n^{-(q-2)/2q}  & \mathrm{if}\:d=4\:\mathrm{and}\:q\neq4\\
n^{-1/d}+n^{-(q-2)/2q} & \mathrm{if}\:d>4\:\mathrm{and}\:q\neq d/(d-2)
\end{cases}.
\end{align}
One can conclude by setting $C_{d, q, \mu, K}\coloneqq 8(C_{d,q,\mu}^{1/2}e_{K, \mu}^{*}\vee C_{d, q, \mu})$.
\end{proof}

\begin{remark}\label{remC1d}
When $d\geq4$, if $\frac{q-2}{q}>\frac{2}{d}$ i.e. $q>\frac{2d}{d-2}$,  Inequality (\ref{w2carre}) can be upper bounded as follows,
\begin{align}
&\mathbb{E}\,\mathcal{W}_{2}^{2}(\mu_{n}^{\omega}, \mu)\leq 2C_{d, q, \mu}n^{-1/d}\times\begin{cases}
%n^{-\frac{1}{2}+\frac{1}{d}}+n^{-\frac{q-2}{q}+\frac{1}{d}} & \mathrm{if}\:d<4\:\mathrm{and}\:q\neq4\\
n^{-\frac{1}{4}}\log(1+n) & \mathrm{if}\:d=4\:\mathrm{and}\:q\neq4\\
n^{-\frac{1}{d}}& \mathrm{if}\:d>4\:\mathrm{and}\:q\neq d/(d-2)
\end{cases}\nonumber\\
&\leq 2C_{d, q, \mu}K^{-1/d}\times\begin{cases}
%n^{-\frac{1}{2}+\frac{1}{d}}+n^{-\frac{q-2}{q}+\frac{1}{d}} & \mathrm{if}\:d<4\:\mathrm{and}\:q\neq4\\
n^{-\frac{1}{4}}\log(1+n) & \mathrm{if}\:d=4\:\mathrm{and}\:q\neq4\\
n^{-\frac{1}{d}} & \mathrm{if}\:d>4\:\mathrm{and}\:q\neq d/(d-2)
\end{cases},\nonumber
\end{align}
since we consider only $n\geq K$ and if $q>\frac{2d}{d-2}$, the term $n^{-(q-2)/2q}$ becomes negligible as $n$ grows. Consequently, (\ref{empriboundK}) can be bounded by 
\begin{align}\label{updesK}
\mathbb{E}\,&\mathcal{D}_{K, \mu}(x^{(n), \omega})-\inf_{x\in(\mathbb{R}^{d})^{K}}\mathcal{D}_{K, \mu}(x)\leq 4e_{K, \mu}^{*}\mathbb{E}\,\mathcal{W}_{2}(\mu_{n}^{\omega}, \mu)+4\,\mathbb{E}\,\mathcal{W}_{2}^{2}(\mu_{n}^{\omega}, \mu).\nonumber\\
&\leq8(C_{d,q,\mu}^{1/2}e_{K, \mu}^{*}\vee 2C_{d, q, \mu}K^{-1/d})\times\nonumber\\
&\;\;\;\begin{cases}
%n^{-\frac{1}{4}}+n^{-\frac{q-2}{2q}}+n^{-\frac{1}{2}+\frac{1}{d}}+n^{-\frac{q-2}{q}+\frac{1}{d}} & \mathrm{if}\:d<4\:\mathrm{and}\:q\neq4\\
n^{-\frac{1}{4}}\big[(\log(1+n))^{\frac{1}{2}}+\log(1+n)\big] & \mathrm{if}\:d=4\:\mathrm{and}\:q\neq4\\
2n^{-\frac{1}{d}} & \mathrm{if}\:d>4\:\mathrm{and}\:q\neq d/(d-2)
\end{cases}.
\end{align}
By the non-asymptotic Zador theorem (\ref{nonasymptoticzador}), one has 
\[e_{K, \mu}^{*}\leq C_{d, q}(\mu)\sigma_{q}(\mu)K^{-1/d}\]
with $\sigma_{q}(\mu)=\min_{a\in\mathbb{R}^{d}}\left[\int_{\mathbb{R}^{d}}\left|\xi-a\right|^{q}\mu(d\xi)\right]^{1/q}$. Thus, %if $d>1$ and $q>\frac{2d}{d-1}$,
Inequality (\ref{updesK}) can be upper-bounded as follows,  
\begin{align}
\mathbb{E}\,&\mathcal{D}_{K, \mu}(x^{(n), \omega})-\inf_{x\in(\mathbb{R}^{d})^{K}}\mathcal{D}_{K, \mu}(x)\leq 4e_{K, \mu}^{*}\mathbb{E}\,\mathcal{W}_{2}(\mu_{n}^{\omega}, \mu)+4\,\mathbb{E}\,\mathcal{W}_{2}^{2}(\mu_{n}^{\omega}, \mu).\nonumber\\
&\leq8K^{-1/d}\big(C_{d,q,\mu}^{1/2}C_{d,q}(\mu)\sigma_{q}(\mu)\vee 2C_{d, q, \mu}\big)\times\nonumber\\
&\;\;\;\begin{cases}
%n^{-\frac{1}{4}}+n^{-\frac{q-2}{2q}}+n^{-\frac{1}{2}+\frac{1}{d}}+n^{-\frac{q-2}{q}+\frac{1}{d}} & \mathrm{if}\:d<4\:\mathrm{and}\:q\neq4\\
n^{-\frac{1}{4}}\big[(\log(1+n))^{\frac{1}{2}}+\log(1+n)\big] & \mathrm{if}\:d=4\:\mathrm{and}\:q\neq4\\
2n^{-\frac{1}{d}} & \mathrm{if}\:d>4\:\mathrm{and}\:q\neq d/(d-2)
\end{cases},\nonumber
\end{align}
from which one can remark that the constant $C_{d, q, \mu, K}$ in Proposition \ref{convrateemprical} is roughly decreasing as $K^{-1/d}$. 
\end{remark}}

A second upper bound of the clustering performance is provided in the following theorem. 
\begin{theorem}\label{performance}
Let $K\in\mathbb{N}^{*}$ be the quantization level. Let $\mu\in\mathcal{P}_{2}(\mathbb{R}^{d})$ with $\text{card}\,(\text{supp}(\mu))\geq K$ and let $\mu_{n}^{\omega}$ be the empirical measures of $\mu$ defined in (\ref{empi}), generated by i.i.d observations $X_{1}, ... , X_{n}, ...$.   We denote by $x^{(n),\omega}\in(\mathbb{R}^{d})^{K}$ an optimal quantizer of $\mu_{n}^{\omega}$ at level $K$. Then, 
\begin{enumerate}[(a)]
\item \emph{General upper bound of the performance}.  
\begin{equation}
\mathbb{E}\,\mathcal{D}_{K,\,\mu}(x^{(n),\omega})-\inf_{x\in(\mathbb{R}^{d})^{K}}\mathcal{D}_{K,\,\mu}(x)\leq \frac{2K}{\sqrt{n}} \Big[r_{2n}^{2}+\rho_{K}(\mu)^{2}+2r_{1}\big(r_{2n}+\rho_{K}(\mu)\big)\Big],
\end{equation}
where  $r_{n}\coloneqq \big\Vert\max_{1\leq i \leq n}\left| X_{i}\right|\big\Vert_{2}$ and $\rho_{K}(\mu)$ is the maximum radius of optimal quantizers of $\mu$, defined in (\ref{maxradius}). 
\item \emph{Asymptotic upper bound for distribution with polynomial tail.}  For $p>2$, if $\mu$ has a $c$-th polynomial tail with $c>d+p$, then 
\[\mathbb{E}\,\mathcal{D}_{K,\,\mu}(x^{(n),\omega})-\inf_{x\in(\mathbb{R}^{d})^{K}}\mathcal{D}_{K,\,\mu}(x)\leq \frac{K}{\sqrt{n}} \Big[C_{\mu, p}\,n^{2/p}+6K^{\frac{2(p+d)}{d(c-p-d)}\gamma_{K}}\Big],\]
where $C_{\mu, p}$ is a constant depending $\mu, p$ and $\lim_{K}\gamma_{K}=1$.
\item \emph{Asymptotic upper bound for distribution with hyper-exponential tail}. Recall that $\mu$ has a hyper-exponential tail if $\mu=f\cdot\lambda_{d}$ and there exists $\tau>0, \kappa,\vartheta>0, c>-d$ and $A>0$ such that $\forall \xi\in\mathbb{R}^{d}, \left|\xi\right|\geq A\Rightarrow\, f(\xi)=\tau \left| \xi\right|^{c}e^{-\vartheta\left|\xi\right|^{\kappa}}$. If $\kappa\geq2$, we can obtain a more precise upper bound of the performance 
\[\mathbb{E}\big[\mathcal{D}_{K,\,\mu}(x^{(n),\omega})-\inf_{x\in(\mathbb{R}^{d})^{K}}\mathcal{D}_{K,\,\mu}(x)\big]\leq C_{\vartheta, \kappa, \mu}\cdot \frac{K}{\sqrt{n}}\Big[ 1+ (\log n)^{2/\kappa}+\gamma_{K}(\log K)^{2/\kappa}\big(1+\frac{2}{d}\big)^{2/\kappa}\Big], \]
where $C_{\vartheta, \kappa, \mu}$ is a constant depending $\vartheta, \kappa, \mu$ and $\limsup_{K}\gamma_{K}=1$. \\
In particular, if $\mu=\mathcal{N}(m, \Sigma)$, the multidimensional normal distribution, we have
\[\mathbb{E}\big[\mathcal{D}_{K,\,\mu}(x^{(n),\omega})-\inf_{x\in(\mathbb{R}^{d})^{K}}\mathcal{D}_{K,\,\mu}(x)\big]\leq C_{\mu}\cdot \frac{K}{\sqrt{n}}\Big[ 1+ \log n+\gamma_{K}\cdot(\log K)\big(1+\frac{2}{d}\big)\Big], \]
where $\limsup_{K}\gamma_{K}=1$ and $C_{\mu}=24\cdot\big(1\vee \log2\mathbb{E}e^{\left| X\right|^{2}/4}\big)$ where $X$ is a random variable with distribution $\mu$. Moreover, when $\mu=\mathcal{N}(0, \mathbf{I}_{d})$, $C_{\mu}=24(1+\frac{d}{2})\cdot\log 2$.
\end{enumerate}
\end{theorem}

The proof of Theorem \ref{performance} relies on the Rademacher process theory. A Rademacher sequence $(\sigma_{i})_{i\in\{1,...,n\}}$ is a sequence of i.i.d  random variables with a symmetric $\{\pm 1\}$-valued Bernoulli distribution, independent of $(X_{1}, ..., X_{n})$ and we define the Rademacher process $\mathcal{R}_{n}(f), f\in\mathcal{F}$ by $\mathcal{R}_{n}(f)\coloneqq \frac{1}{n}\sum_{i=1}^{n}\sigma_{i}f(X_{i})$. Remark that  the Rademacher process $\mathcal{R}_{n}(f)$ depends on the sample $\{X_{1},..., X_{n}\}$ of the probability measure $\mu$. 

\begin{thm*}[Symmetrization inequalites]
For any class $\mathcal{F}$ of $\mu$-integrable functions, we have \[\mathbb{E}\left\Vert\mu_{n}-\mu\right\Vert_{\mathcal{F}}\leq 2 \mathbb{E}\left\Vert \mathcal{R}_{n}\right\Vert_{\mathcal{F}},\] 
where for a probability distribution $\nu$, $\left\Vert\nu\right\Vert_{\mathcal{F}}\coloneqq\sup_{f\in\mathcal{F}}\left|\nu(f)\right|\coloneqq\sup_{f\in\mathcal{F}}\left|\int_{\mathbb{R}^{d}}fd\nu\right|$ and $\left\Vert\mathcal{R}_{n}\right\Vert_{\mathcal{F}}\coloneqq \sup_{f
\in\mathcal{F}}\left| \mathcal{R}_{n}(f)\right|$. 
\end{thm*}
For the proof of the above theorem, we refer to \cite{koltchinskii2011introduction}[see Theorem 2.1]. Another more detailed reference is \cite{van1996weak}[see Lemma 2.3.1]. We will also introduce the \textit{Contraction principle} in the following theorem and we refer to \cite{boucheron2013concentration}[see Theorem 11.6] for the proof. 

\begin{thm*}[Contraction principle]
Let $x_{1}, ..., x_{n}$ be vectors whose real-valued components are indexed by $\mathcal{T}$, that is, $x_{i}=(x_{i,s})_{s\in\mathcal{T}}$. For each $i=1, ..., n$, let $\varphi_{i}: \mathbb{R}\rightarrow \mathbb{R}$ be a Lipschitz function such that $\varphi_{i}(0)=0$. Let $\sigma_{1}, ..., \sigma_{n}$ be independent Rademacher random variables and let $c_{L}=\max_{1\leq i\leq n}\sup_{\substack{x,y\in\mathbb{R}\\ x\neq y}}\left|\frac{\varphi_{i}(x)-\varphi_{i}(y)}{x-y}\right|$ be the uniform Lipschitz constant of the function $\varphi_{i}$. Then 
\begin{equation}\label{contraprin}
\mathbb{E}\Big[ \sup_{s\in\mathcal{T}}\sum_{i=1}^{n}\sigma_{i}\varphi_{i}(x_{i,s})\Big]\leq c_{L}\cdot\mathbb{E}\Big[\sup_{s\in\mathcal{T}}\sum_{i=1}^{n}\sigma_{i}x_{i,s}\Big].
\end{equation}
\end{thm*}

Remark that, if we consider  random variables $(Y_{1, s}, ..., Y_{n, s})_{s\in\mathcal{T}}$ independent of $(\sigma_{1}, ..., \sigma_{n})$ and for all $s\in\mathcal{T}$ and $i\in\{1, ..., n\}$, $Y_{i, s}$ is valued in $\mathbb{R}$, then (\ref{contraprin}) implies that 
\begin{align}\label{contraction}
\mathbb{E}\Big[ \sup_{s\in\mathcal{T}}\sum_{i=1}^{n}&\sigma_{i}\varphi_{i}(Y_{i,s})\Big]=\mathbb{E}\Big\{\mathbb{E}\Big[ \sup_{s\in\mathcal{T}}\sum_{i=1}^{n}\sigma_{i}\varphi_{i}(Y_{i,s})\mid (Y_{1, s}, ..., Y_{n, s})_{s\in\mathcal{T}}\Big]\Big\}\nonumber\\
\leq & c_{L}\cdot \mathbb{E}\Big\{\mathbb{E}\Big[ \sup_{s\in\mathcal{T}}\sum_{i=1}^{n}\sigma_{i}Y_{i,s}\mid (Y_{1, s}, ..., Y_{n, s})_{s\in\mathcal{T}}\Big]\Big\}\leq c_{L}\cdot \mathbb{E}\Big[ \sup_{s\in\mathcal{T}}\sum_{i=1}^{n}\sigma_{i}Y_{i,s}\big].
\end{align}
The proof of Theorem \ref{performance} is  inspired by that of Theorem 2.1 in \cite{biau2008performance}.
\begin{proof}[Proof of Theorem \ref{performance}] 
(a) In order to simplify the notation, we will denote by $\mathcal{D}$ (\textit{respectively} $\mathcal{D}_{n}$) instead of $\mathcal{D}_{K,\,\mu}$ (\textit{resp.} $\mathcal{D}_{K,\,\mu_{n}}$)  the distortion function of $\mu$ (\textit{resp.} $\mu_{n}$). For any $c=(c_{1}, ..., c_{K})\in(\mathbb{R}^{d})^{K}$, note that the distortion function $\mathcal{D}(c)$ of $\mu$ can be written as 
\[\mathcal{D}(c)=\mathbb{E}\big[\min_{1\leq k\leq K} \left| X-c_{k}\right|^{2}\big]=\mathbb{E}\big[\left|X\right|^{2}+\min_{1\leq k\leq K}(-2\langle X | c_{k}\rangle+\left|c_{k}\right|^{2})\big].\]
We define $\overline{\mathcal{D}}(c)\coloneqq \min_{1\leq k\leq K} \big( -2\langle X | c_{k}\rangle +\left|c_{k}\right|^{2}\big)$. Similarly, for the distortion function $\mathcal{D}_{n}$ of the empirical measure $\mu_{n}$, 
\[\mathcal{D}_{n}(c)=\frac{1}{n}\sum_{i=1}^{n}\min_{1\leq k\leq K}\left| X_{i}-c_{k}\right|^{2}=\frac{1}{n}\sum_{i=1}^{n}\left|X_{i}\right|^{2}+\min_{1\leq k\leq K}\big(-\frac{2}{n}\sum_{i=1}^{n}\langle X_{i} | c_{k}\rangle+\left|c_{k}\right|^{2}\big),\]
we define $\overline{\mathcal{D}}_{n}(c)\coloneqq \min_{1\leq k\leq K}\big(-\frac{2}{n}\sum_{i=1}^{n}\langle X_{i} | c_{k}\rangle+\left|c_{k}\right|^{2}\big)$. We will drop $\omega$ in $x^{(n), \omega}$ to alleviate the notation throughout the proof. Let $x\in \mathrm{argmin}\,\mathcal{D}_{K, \mu}$.
It follows that 
\begin{align}\label{25}
\mathbb{E}\big[\mathcal{D}(x^{(n)})-\mathcal{D}(x)\big] &= \mathbb{E}\big[\overline{\mathcal{D}}(x^{(n)})-\overline{\mathcal{D}}(x)\big]=\mathbb{E}\big[\overline{\mathcal{D}}(x^{(n)})-\overline{\mathcal{D}}_{n}(x^{(n)})\big]+\mathbb{E}\big[\overline{\mathcal{D}}_{n}(x^{(n)})-\overline{\mathcal{D}}(x)\big]\nonumber\\
&\leq\mathbb{E}\big[\overline{\mathcal{D}}(x^{(n)})-\overline{\mathcal{D}}_{n}(x^{(n)})\big]+\mathbb{E}\big[\overline{\mathcal{D}}_{n}(x)-\overline{\mathcal{D}}(x)\big].
\end{align}
Define for $\eta, x\in\mathbb{R}^{d}$, $f_{\eta}(x)=-2\langle \eta | x\rangle+\left| \eta\right|^{2}$. 

\smallskip
\noindent\textit{Part (i): Upper bound of $\mathbb{E}[\overline{\mathcal{D}}(x^{(n)})-\overline{\mathcal{D}}_{n}(x^{(n)})]$}. Let $R_{n}(\omega)\coloneqq \max_{1\leq i\leq n}\left| X_{i}(\omega)\right|$. Remark that for every $\omega\in\Omega$, $R_{n}(\omega)$ is invariant with the respect to all permutations of the components of $(X_{1}, ..., X_{n})$. Let $B_{R}$ denote the ball centred at 0 with radius $R$. Then, owing to Proposition \ref{relation}-(iii), $x^{(n)}=(x^{(n)}_{1}, ..., x^{(n)}_{K})\in B_{R_{n}}^{K}$. Hence, 
\begin{align}\label{26}
\mathbb{E}\,[\overline{\mathcal{D}}(x^{(n)})-&\overline{\mathcal{D}}_{n}(x^{(n)})]\leq \mathbb{E}\,\sup_{c\,\in B_{R_{n}}^{K}}\big(\overline{\mathcal{D}}(c)-\overline{\mathcal{D}}_{n}(c)\big)\nonumber\\
=&\; \mathbb{E}\,\big[\sup_{c\,\in B_{R_{n}}^{K}}\big( \mathbb{E}\min_{1\leq k\leq K}f_{c_{k}}(X) - \frac{1}{n}\sum_{i=1}^{n}\min_{1\leq k\leq K}f_{c_{k}}(X_{i})\big)\big]\nonumber\\
=&\; \mathbb{E}\,\Big{[}\sup_{c\,\in B_{R_{n}}^{K}}\mathbb{E} \big[ \frac{1}{n}\sum_{i=1}^{n} \min_{1\leq k\leq K} f_{c_{k}}(X'_{i}) - \frac{1}{n}\sum_{i=1}^{n} \min_{1\leq k\leq K} f_{c_{k}}(X_{i})   \big|X_{1}, ..., X_{n}\big]\Big{]},
\end{align}
where $X'_{1}, ..., X'_{n}$ are i.i.d random variable with the distribution $\mu$, independent of $(X_{1}, ..., X_{n})$. Let $R_{2n}\coloneqq \max_{1\leq i \leq n}\left|X_{i}\right|\vee\left|X'_{i}\right|$, then  (\ref{26}) becomes
\begin{align}
\mathbb{E}\,[\overline{\mathcal{D}}(x^{(n)})-&\overline{\mathcal{D}}_{n}(x^{(n)})]\leq \mathbb{E}\,\Big{[}\sup_{c\,\in B_{R_{2n}}^{K}}\mathbb{E} \big[ \frac{1}{n}\sum_{i=1}^{n} \min_{1\leq k\leq K} f_{c_{k}}(X'_{i}) - \frac{1}{n}\sum_{i=1}^{n} \min_{1\leq k\leq K} f_{c_{k}}(X_{i})   \big|X_{1}, ..., X_{n}\big]\Big{]}\nonumber\\
\leq &\; \mathbb{E}\,\Big{[}\mathbb{E}\big[ \sup_{c\,\in B_{R_{2n}}^{K}} \big( \frac{1}{n}\sum_{i=1}^{n} \min_{1\leq k\leq K} f_{c_{k}}(X'_{i}) - \frac{1}{n}\sum_{i=1}^{n} \min_{1\leq k\leq K} f_{c_{k}}(X_{i}) \big)  \big|X_{1}, ..., X_{n}\big]\Big{]}\nonumber\\
= & \; \mathbb{E} \big[  \sup_{c\,\in B_{R_{2n}}^{K}}  \frac{1}{n}\sum_{i=1}^{n} \big( \min_{1\leq k\leq K} f_{c_{k}}(X'_{i}) - \min_{1\leq k\leq K} f_{c_{k}}(X_{i}) \big)  \big].
\end{align}
The distribution of $(X_{1}, ..., X_{n}, X'_{1}, ..., X'_{n})$ and that of $R_{2n}$ are invariant with the respect to all permutation of the components in $(X_{1}, ..., X_{n}, X'_{1}, ..., X'_{n})$. Hence, 
\begin{align}\label{sigmasign}
\mathbb{E}\,[\overline{\mathcal{D}}(x^{(n)})-&\overline{\mathcal{D}}_{n}(x^{(n)})] = \mathbb{E}\big[ \sup_{c\,\in B_{R_{2n}}^{K}}\frac{1}{n}\sum_{i=1}^{n}\sigma_{i}\big( \min_{1\leq k\leq K} f_{c_{k}}(X'_{i}) - \min_{1\leq k\leq K} f_{c_{k}}(X_{i}) \big) \big] \nonumber\\
\leq &\; \mathbb{E}\big[\sup_{c\,\in B_{R_{2n}}^{K}} \frac{1}{n} \sum_{i=1}^{n}\sigma_{i}\min_{1\leq k\leq K} f_{c_{k}}(X'_{i})\big] +\mathbb{E}\big[\sup_{c\,\in B_{R_{2n}^{K}}} \frac{1}{n} \sum_{i=1}^{n}\sigma_{i}\min_{1\leq k\leq K} f_{c_{k}}(X_{i})\big] \nonumber\\
=&\; 2\mathbb{E}\big[\sup_{c\,\in B_{R_{2n}}^{K}} \frac{1}{n} \sum_{i=1}^{n}\sigma_{i}\min_{1\leq k\leq K} f_{c_{k}}(X_{i})\big].
\end{align}
In the second line of (\ref{sigmasign}), we can change the sign before the second term since $-\sigma_{i}$ has the same distribution of $\sigma_{i}$, and we will continue to use this property throughout the proof. Let $\displaystyle S_{K}=\mathbb{E}\Big[\sup_{c\,\in B_{R_{2n}}^{K}} \frac{1}{n} \sum_{i=1}^{n}\sigma_{i}\min_{1\leq k\leq K} f_{c_{k}}(X_{i})\Big]$ {\color{black}and we provide an upper bound for $S_{K}$ by induction on $K$ in what follows. }

\vspace{0.3cm}
$\blacktriangleright$ For ${K}=1$,
\begin{align}\label{29}
S_{1} =&\; \mathbb{E}\big[ \sup_{c\,\in B_{R_{2n}}}\frac{1}{n} \sum_{i=1}^{n}\sigma_{i}\min_{1\leq k\leq K} f_{c}(X_{i})\big]=\mathbb{E}\big[\sup_{c\,\in B_{R_{2n}}}\frac{1}{n}\sum_{i=1}^{n}\sigma_{i}\big( -2\langle c | X_{i}\rangle+\left|c\right|^{2}\big)\big]\nonumber\\
\leq&\;2\,\mathbb{E}\big[\sup_{c\,\in B_{R_{2n}}}\frac{1}{n}\sum_{i=1}^{n}\sigma_{i}\langle c | X_{i}\rangle\big]+\mathbb{E}\big[\sup_{c\,\in B_{R_{2n}}}\frac{1}{n}\sum_{i=1}^{n}\sigma_{i}\left|c\right|^{2}\big]\nonumber\\
\leq&\; \frac{2}{n}\mathbb{E}\Big[\sup_{c\,\in B_{R_{2n}}}\langle c | \sum_{i=1}^{n}\sigma_{i} X_{i}\rangle\Big]+\frac{1}{n}\mathbb{E}\Big[ \left| \sum_{i=1}^{n}\sigma_{i}\right|\cdot\left|R_{2n}\right|^{2}\Big]\nonumber\\
\leq&\; \frac{2}{n}\mathbb{E}\Big[\sup_{c\,\in B_{R_{2n}}}\left|\sum_{i=1}^{n}\sigma_{i}X_{i}\right|\cdot\left|c\right|\Big]+\frac{1}{n}\mathbb{E} \left| \sum_{i=1}^{n}\sigma_{i}\right|\cdot\mathbb{E}\left|R_{2n}\right|^{2} \nonumber\\
&\text{(by Cauchy-Schwarz inequality and independence of $\sigma_{i}$ and $X_{i}$)}\nonumber\\
\leq&\; \frac{2}{n}\left\Vert\sum_{i=1}^{n}\sigma_{i}X_{i}\right\Vert_{2}\cdot \left\Vert R_{2n}\right\Vert_{2}+\frac{1}{n} \left\Vert\sum_{i=1}^{n}\sigma_{i}\right\Vert_{2}^{2}\cdot \left\Vert R_{2n}\right\Vert_{2}^{2}\nonumber\\
\leq&\; \frac{2}{n}\sqrt{n}\left\Vert X_{1}\right\Vert_{2}\cdot\left\Vert R_{2n}\right\Vert_{2}+\frac{1}{\sqrt{n}}\left\Vert R_{2n}\right\Vert_{2}^{2}\leq \frac{\left\Vert R_{2n}\right\Vert_{2}}{\sqrt{n}}\big(2\left\Vert X_{1}\right\Vert_{2}+\left\Vert R_{2n}\right\Vert_{2}\big).
\end{align}
The first inequality of the last line of  (\ref{29}) follows from $\mathbb{E}\left|\sum_{i=1}^{n}\sigma_{i}X_{i}\right|^{2}=\mathbb{E}\sum_{i=1}^{n}\sigma_{i}^{2}X_{i}^{2}=n\mathbb{E}X_{1}^{2}$ since the $(\sigma_{1}, ..., \sigma_{n})$ is independent of $(X_{1}, ..., X_{n})$ and $\mathbb{E}\,\sigma_{i}=0$. For $n\in\mathbb{N}^{*}$, define $r_{n}\coloneqq \left\Vert \max_{1\leq i\leq n} \left| Y_{i}\right|\right\Vert_{2}$, where $Y_{1}, ..., Y_{n}$ are i.i.d random variables with probability distribution $\mu$. Hence, $r_{2n}=\left\Vert R_{2n}\right\Vert_{2}$, since $(Y_{1}, ..., Y_{2n})$ has the same distribution as $(X_{1}, ..., X_{n}, X'_{1}, ..., X'_{n})$. Therefore, 
\[S_{1}\leq \frac{r_{2n}}{\sqrt{n}}\big(2\left\Vert X_{1}\right\Vert_{2}+r_{2n}\big).\]

$\blacktriangleright$ For $K=2$, 
\begin{align}
S_{2}=&\;\mathbb{E}\big[\sup_{c=(c_{1}, c_{2})\in B_{R_{2n}}^{2}}\frac{1}{n}\sum_{i=1}^{n}\sigma_{i}\big(f_{c_{1}}(X_{i}) \wedge f_{c_{2}}(X_{i})\big)\big]\nonumber\\
=&\;\frac{1}{2}\mathbb{E}\Big[ \sup_{c\,\in B_{R_{2n}}^{2}}\frac{1}{n}\sum_{i=1}^{n}\sigma_{i}\big(f_{c_{1}}(X_{i}) +f_{c_{2}}(X_{i})-\left|f_{c_{1}}(X_{i}) -f_{c_{2}}(X_{i}) \right| \big)\Big] (\text{as $a\wedge b=\frac{a+b}{2}-\frac{\left|a-b\right|}{2}$})\nonumber\\
\leq &\;\frac{1}{2}\Big{\{}\mathbb{E}\big[\sup_{c\,\in B_{R_{2n}}^{2}} \frac{1}{n}\sum_{i=1}^{n}\sigma_{i}\big(f_{c_{1}}(X_{i}) +f_{c_{2}}(X_{i})\big)\big] +\mathbb{E}\big[\sup_{c\,\in B_{R_{2n}}^{2}} \frac{1}{n}\sum_{i=1}^{n}\sigma_{i}\left|f_{c_{1}}(X_{i})-f_{c_{2}}(X_{i})\right|\big]\Big{\}}\nonumber\\
\leq & \;\frac{1}{2}\Big\{2S_{1}+\mathbb{E}\big[\sup_{c\,\in B_{R_{2n}}^{2}} \frac{1}{n}\sum_{i=1}^{n}\sigma_{i}\big(f_{c_{1}}(X_{i})-f_{c_{2}}(X_{i})\big)\big]\Big\}\;\;\;\text{\big(by (\ref{contraction})\big)}\nonumber\\
\leq &\; \frac{1}{2}\Big\{2S_{1}+\mathbb{E}\big[\sup_{c_{1}\in B_{R_{2n}}}\frac{1}{n}\sum_{i=1}^{n}\sigma_{i}f_{c_{1}}(X_{i})\big]+\mathbb{E}\big[\sup_{c_{2}\in B_{R_{2n}}}\frac{1}{n}\sum_{i=1}^{n}\sigma_{i}f_{c_{2}}(X_{i})\big] \Big\}\leq 2S_{1}.
\end{align}

$\blacktriangleright$ Next, we will show by induction that $S_{K}\leq KS_{1}$ for every $K\in\mathbb{N}^{*}$. Assume that $S_{K}\leq KS_{1}$, for $K+1$,
\begin{align}
S_{K+1}= &\;\mathbb{E}\big[\sup_{c\,\in B_{R_{2n}}^{K+1}}\frac{1}{n}\sum_{i=1}^{n}\sigma_{i}\min_{1\leq k\leq K+1} f_{c_{k}}(X_{i})\big]\nonumber\\
=& \;\mathbb{E}\big[\sup_{c\,\in B_{R_{2n}}^{K+1}}\frac{1}{n}\sum_{i=1}^{n}\sigma_{i}\big(\min_{1\leq k\leq K} f_{c_{k}}(X_{i})\wedge f_{c_{K+1}}(X_{i})\big)\big]\nonumber\\
\leq&\;\frac{1}{2}  \mathbb{E}\Big\{ \sup_{c\,\in B_{R_{2n}}^{K+1}} \frac{1}{n}\sum_{i=1}^{n}\sigma_{i} \Big[\big(\min_{1\leq k\leq K} f_{c_{k}}(X_{i})+ f_{c_{K+1}}(X_{i})\big)-\left|\min_{1\leq k\leq K} f_{c_{k}}(X_{i})- f_{c_{K+1}}(X_{i})\right|  \Big]\Big\}\nonumber\\
\leq &\; \frac{1}{2} \mathbb{E}\Big\{\sup_{c\,\in B_{R_{2n}}^{K+1}} \frac{1}{n}\sum_{i=1}^{n}\sigma_{i} \big(\min_{1\leq k\leq K} f_{c_{k}}(X_{i})+ f_{c_{K+1}}(X_{i})\big)\nonumber\\&\;\;\;\;+ \sup_{c\,\in B_{R_{2n}}^{K+1}}\frac{1}{n}\sum_{i=1}^{n}\sigma_{i}\left|\min_{1\leq k\leq K} f_{c_{k}}(X_{i})- f_{c_{K+1}}(X_{i})\right| \Big\}\nonumber\\
\leq &\; \frac{1}{2}(S_{K}+S_{1}+S_{K}+S_{1})\leq S_{K}+S_{1}\leq (K+1)S_{1}. 
\end{align}
Hence,
\begin{equation}
\mathbb{E}\,[\overline{\mathcal{D}}(x^{(n)})-\overline{\mathcal{D}}_{n}(x^{(n)})] \leq 2S_{K}\leq 2KS_{1}\leq \frac{2K\cdot r_{2n}}{\sqrt{n}}\big(2\left\Vert X_{1}\right\Vert_{2}+r_{2n}\big).
\end{equation}

\noindent\textit{Part (ii): Upper bound of $\mathbb{E}\,[\overline{\mathcal{D}}_{n}(x)-\overline{\mathcal{D}}(x)]$}. As $x=(x_{1}, ..., x_{K})$ is an optimal  quantizer of $\mu$, we have $\max_{1\leq k \leq K}\left| x_{k}\right|\leq \rho_{K}(\mu)$ owing to the definition of $\rho_{K}(\mu)$ in (\ref{maxradius}). Consequently, 
\[\mathbb{E}\big[\overline{\mathcal{D}}_{n}(x)-\overline{\mathcal{D}}(x)\big]\leq \mathbb{E}\sup_{c\,\in B_{\rho_{K}(\mu)}^{K}}\big[\overline{\mathcal{D}}_{n}(c)-\overline{\mathcal{D}}(c)\big]\]
By the same reasoning of Part (I), we have $\mathbb{E}\big[\overline{\mathcal{D}}_{n}(x)-\overline{\mathcal{D}}(x)\big]\leq \frac{2K}{\sqrt{n}}\rho_{K}(\mu)\big(2\left\Vert X_{1}\right\Vert_{2}+\rho_{K}(\mu)\big)$.
Hence 
\begin{align}\label{geneperformance}
\mathbb{E}\big[\mathcal{D}(x^{(n)})-\mathcal{D}(x)\big]&\leq\frac{2K}{\sqrt{n}}r_{2n}\big(2\left\Vert X_{1}\right\Vert_{2}+r_{2n}\big)+\frac{2K}{\sqrt{n}}\rho_{K}(\mu)\big(2\left\Vert X_{1}\right\Vert_{2}+\rho_{K}(\mu)\big)\nonumber\\
&\leq\frac{2K}{\sqrt{n}}\Big[r_{2n}^{2}+\rho^{2}_{K}(\mu)+2r_{1}\big(r_{2n}+\rho_{K}(\mu)\big)\Big].\end{align}

The proof of $(b)$ and $(c)$ is postponed in Appendix E. 
\end{proof}

\smallskip
\section{Appendix}

\subsection{Appendix A:  Proof of Pollard's Theorem}
\begin{proof}[Proof of Pollard's Theorem]
Since the quantization level $K$ is fixed, in this proof, we drop the  subscript $K$ of the distortion function and denote by $\mathcal{D}_{n}$ (\textit{respectively,} $\mathcal{D}_{\infty})$  the distortion function of $\mu_{n}$ (\textit{resp.} $\mu_{\infty}$).

We know $x^{(n)}\in\mathrm{argmin}\;\mathcal{D}_{n}$ owing to Proposition \ref{relation}, that is, for all $y\in(y_{1}, ..., y_{K})\in(\mathbb{R}^{d})^{K}$, we have $\mathcal{D}_{n}(x^{(n)})\leq \mathcal{D}_{n}(y)$. For every fixed $y=(y_{1}, ..., y_{K})$, we have $\mathcal{D}_{n}(y)\rightarrow\mathcal{D}_{\infty}(y)$ after (\ref{06}) so that 
\begin{equation}\label{limsupone}
\limsup_{n}\;\mathcal{D}_{n}(x^{(n)})\leq \inf_{y\in(\mathbb{R}^{d})^{K}}\mathcal{D}_{\infty}(y).
\end{equation}

Assume that there exists an index set $\mathcal{I}\;{\subset}\;\{1,... ,K\}$ and $\mathcal{I}^{c}\neq\varnothing$ such that $(x_{i}^{(n)})_{i\in\mathcal{I}, n\geq 1}$ is bounded and $(x_{i}^{(n)})_{i\in\mathcal{I}^{c}, n\geq 1}$ is not bounded. Then there exists a subsequence $\psi(n)$ of $n$ such that \[\begin{cases}
x_{i}^{\psi(n)}\rightarrow \widetilde{x}_{i}^{(\infty)}, & i\in\mathcal{I},\\
\left|x_{i}^{\psi(n)}\right|\rightarrow +\infty, & i\in\mathcal{I}^{c}.
\end{cases}\]

After (\ref{06}), we have $\mathcal{D}_{\psi(n)}(x^{(\psi(n))})^{1/2}\geq \mathcal{D}_{\infty}(x^{(\psi(n))})^{1/2}-\mathcal{W}_{2}(\mu_{\psi(n)}, \mu_{\infty})$. Hence,
\[\liminf_{n}\mathcal{D}_{\psi(n)}(x^{(\psi(n))})^{1/2}\geq \liminf_{n}\mathcal{D}_{\infty}(x^{(\psi(n))})^{1/2}\]
so that
\begin{align}\label{thisone}
\liminf_{n}\,\mathcal{D}_{\psi(n)}(x^{(\psi(n))})^{1/2}\geq&\liminf_{n}\mathcal{D}_{\infty}(x^{(\psi(n))})^{1/2}\nonumber\\=& \big[\liminf_{n} \int \min_{i\in\{1,...,K\}}\left|x_{i}^{(\psi(n))}-\xi\right|^{2}\mu_{\infty}(d\xi)\big]^{1/2} \nonumber\\
\geq&\big[\int\liminf_{n}\min_{i\in\{1,...,K\}}\left|x_{i}^{(\psi(n))}-\xi\right|^{2}\mu_{\infty}(d\xi)\big]^{1/2}\nonumber\\
=&\big[\int \min_{i\in\mathcal{I}}\left|x_{i}^{(\infty)}-\xi\right|^{2}\mu_{\infty}(d\xi)\big]^{1/2},
\end{align}
where we used Fatou's Lemma in the third line. 
Thus, (\ref{limsupone}) and (\ref{thisone}) imply that 
\begin{equation}\label{otherwise}
\int \min_{i\in\mathcal{I}}\left|x_{i}^{(\infty)}-\xi\right|^{2}\mu_{\infty}(d\xi)\leq\inf_{y\in(\mathbb{R}^{d})^{K}}\mathcal{D}_{\infty}(y).
\end{equation}
This implies that $\mathcal{I}=\{1,... ,K\}$ after Proposition \ref{relation} ({otherwise, (\ref{otherwise}) implies that $e^{\left|\mathcal{I}\right|, *}(\mu_{\infty})\leq e^{K, *}(\mu_{\infty})$ with $\left|\mathcal{I}\right|<K$, which is contradictory to Proposition \ref{relation}-(i)}). Therefore, $(x^{(n)})$ is bounded and any limiting point $x^{(\infty)}\in \mathrm{argmin}_{x\in(\mathbb{R}^{d})^{K}}\mathcal{D}_{\infty}(x)$. \end{proof}

\subsection{Appendix B:  Proof of Proposition \ref{relation} - (iii)}

{\color{black}We define the \textit{open Vorono\"i cell} generated by $x_{i}$ with respect to the Euclidean norm $|\cdot|$ by 
\begin{equation}\label{voroopendef}
V^{o}_{x_{i}}(x)=\big\{\xi\in\mathbb{R}^{d}\,\,\big|\,\left|\xi-x_{i}\right|<\min_{1\leq j\leq K, j\neq i}\left|\xi-x_{j}\right|\big\}.
\end{equation}
It follows from \cite{graf2000foundations}[see Proposition 1.3] that $\mathrm{int}V_{x_{i}}(x)=V_{x_{i}}^{o}(x)$, where $\mathrm{int}A$ denotes the interior of a set $A$. Moreover, if we denote by $\lambda_{d}$ the Lebesgue measure on $\mathbb{R}^{d}$, we have $\lambda_{d}\big(\partial V_{x_{i}}(x)\big)=0$, where $\partial A$ denotes the boundary of $A$ (see \cite{graf2000foundations}[Theorem 1.5]). If $\mu\in\mathcal{P}_{2}(\mathbb{R}^{d})$ and $x^{*}$ is an optimal quantizer of $\mu$, even if $\mu$ is not absolutely continuous with the respect of $\lambda_{d}$, we have $\mu\big(\partial V_{x_{i}}(x^{*})\big)=0$ for all $i\in\{1, ..., K\}$ (see \cite{graf2000foundations}[Theorem 4.2]).}

\begin{proof}
Assume that there exists an $x^{*}=(x_{1}^{*}, ..., x_{K}^{*})\in G_{K}(\mu)$ in which there exists $k\in\{1, ..., K\}$ such that $x_{k}^{*}\notin\mathcal{H}_{\mu}$. %Let $x^{*}=(x_{1}^{*}, ..., x_{K}^{*})$.

\smallskip
\noindent \textit{Case (I): $\mu\big(V_{x_{k}^{*}}^{o}(\Gamma^{*}) \cap \mathrm{supp}(\mu)\big)=0$.}
%After (\ref{rewriteg}), 
The distortion function can be written as 
\begin{align}
\mathcal{D}_{K,\,\mu}(x^{*})&=\sum_{i=1}^{K}\int_{C_{x_{i}}(x)}\left|\xi-x^{*}_{i}\right|^{2}\mu(d\xi)=\sum_{i=1}^{K}\int_{V^{o}_{x_{i}}(x)}\left|\xi-x^{*}_{i}\right|^{2}\mu(d\xi)\nonumber\\
&\text{(since $x^{*}$ is optimal and $\left|\cdot\right|$ is Euclidean, $\mu\big(\partial V_{x_{i}}(\Gamma^{*})\big)=0$ and $\mathrm{int} V_{x_{i}}(\Gamma)=V_{x_{i}}^{o}(\Gamma)$)}\nonumber\\
&=\sum_{i=1, i\neq k}^{K}\int_{V^{o}_{x_{i}}(x)}\left|\xi-x^{*}_{i}\right|^{2}\mu(d\xi)=\mathcal{D}_{K,\,\mu}(\widetilde{x}),
\end{align}
where $\widetilde{x}=(x_{1}^{*}, ..., x_{k-1}^{*}, x_{k+1}^{*}, ..., x_{K}^{*})$. Therefore, $\widetilde{\Gamma}=\{x_{1}^{*}, ..., x_{k-1}^{*}, x_{k+1}^{*}, ..., x_{K}^{*}\}$ is also a $K$-level optimal  quantizer with $\mathrm{card}(\widetilde{\Gamma})<K$, contradictory to Proposition \ref{relation} - (i).

\smallskip
\noindent \textit{Case (II): $\mu\big(V_{x_{k}^{*}}^{o}(\Gamma^{*})\cap\mathrm{supp}(\mu)\big)>0$.} Since $x_{k}^{*}\neq \mathcal{H}_{\mu}$, there exists a hyperplane $H$ strictly {separating} $x_{k}^{*}$ and ${\mathcal{H}_{\mu}}$. Let $\hat{x}_{k}^{*}$ be the orthogonal projection of $x_{k}^{*}$ on $H$. For any $z\in \mathcal{H}_{\mu}$, let $b$ denote {the point} in the segment joining $z$ and $x_{k}^{*}$ which lies on $H$, then $\langle b-\hat{x}_{k}^{*} | x_{k}^{*}-\hat{x}_{k}^{*}\rangle=0$. Hence, 
\[\left|x_{k}^{*}-b\right|^{2}=\left|\hat{x}_{k}^{*}-b\right|^{2}+\left|x_{k}^{*}-\hat{x}_{k}^{*}\right|^{2}>\left|\hat{x}_{k}^{*}-b\right|^{2}.\]
Therefore, {$\left|z-\hat{x}_{k}^{*}\right|\leq \left| z-b\right|+\left|b-\hat{x}_{k}^{*}\right|<\left| z-b\right|+\left|x_{k}^{*}-b\right|=\left|z-x_{k}^{*}\right|$}. 

Let $B(x, r)$ denote the ball on $\mathbb{R}^{d}$ centered at $x$ with radius $r$. Since $\mu\big(V_{x_{k}^{*}}^{o}(\Gamma^{*})\cap\mathrm{supp}(\mu)\big)>0$, there exists $\alpha\in V_{x_{k}^{*}}^{o}(\Gamma^{*})\cap\mathrm{supp}(\mu)$ such that $\exists\, r\geq 0$, $\mu\big(B(\alpha, r)\big)>0$ (when $r=0$, $B(\alpha, r)=\{r\}$). Moreover, 
\begin{equation}\label{label}
\forall \beta\in B(\alpha, r),\quad \left|\beta-\hat{x}_{k}^{*}\right|<\left|\beta-x_{k}^{*}\right|<\min_{i\neq k}\left|\beta-\hat{x}_{i}^{*}\right|. 
\end{equation}
Let $\hat{x}\coloneqq (x_{1}^{*}, ..., x_{k-1}^{*},\hat{x}_{k}^{*}, x_{k+1}^{*}, ..., x_{K}^{*})$, (\ref{label}) implies $\mathcal{D}_{K,\,\mu}(\hat{x})<\mathcal{D}_{K,\,\mu}(x^{*})$. This is contradictory with the fact that $x^{*}$ is an optimal quantizer. Hence, $x^{*}\in\mathcal{H}_{\mu}$.
\end{proof}

\subsection{Appendix C:  Proof of Proposition \ref{contihg}}\label{appenC}

We use Lemma 11 in \cite{fort1995convergence} to compute the Hessian matrix $H_{\mathcal{D}_{K,\,\mu}}$ of $\mathcal{D}_{K,\,\mu}$. 
\begin{lemma}[Lemma 11 in \cite{fort1995convergence}]\label{lemmajean}
Let $\varphi$ be a countinous $\mathbb{R}$-valued function defined on $[0,1]^{d}$. For every $x\in D_{K}\coloneqq\big\{y\in\big([0,1]^{d}\big)^{K}\mid y_{i}\neq y_{j} \;\text{if}\; i\neq j \big\}$, let $\Phi_{i}(x)\coloneqq \int_{V_{i}(x)}\varphi(\omega)d\omega$. Then $\Phi_{i}$ is continuously differentiable on $D_{K}$ and 
\begin{align}
\forall i\neq j, \;\;&\frac{\partial\Phi_{i}}{\partial x_{j}}(x)=\int_{V_{i}(x)\cap V_{j}(x)}\varphi(\xi)\big\{\frac{1}{2}\overrightarrow{n}_{x}^{ij}+\frac{1}{\left| x_{j}-x_{i}\right|}\times (\frac{x_{i}+x_{j}}{2}-\xi)\big\}\lambda_{x}^{ij}(d\xi)\\
\text{and}\;\;&\frac{\partial\Phi_{i}}{\partial x_{i}}(x)=-\sum_{1\leq j\leq K, j\neq i}\frac{\partial\Phi_{j}}{\partial x_{i}}(x),
\end{align}
where $\overrightarrow{n}_{x}^{ij}\coloneqq \frac{x_{j}-x_{i}}{\left| x_{j}-x_{i}\right|}$, 
\begin{equation}\label{midplan}
M^{x}_{ij}\coloneqq\Big\{u\in\mathbb{R}^{d}\mid \langle u-\frac{x_{i}+x_{j}}{2}\mid x_{i}-x_{j}\rangle=0\Big\}
\end{equation} 
and $\lambda_{x}^{ij}(d\xi)$ denotes the Lebesgue measure on the affine hyperplane $M^{x}_{ij}$. 
\end{lemma}
Note that one can simplify the result of Lemma \ref{lemmajean} as follows, 
\begin{align}\label{simplifylemma}
\forall i\neq j, \;\;\frac{\partial\Phi_{i}}{\partial x_{j}}(x)&=\int_{V_{i}(x)\cap V_{j}(x)}\varphi(\xi)\big\{\frac{1}{2}\frac{x_{j}-x_{i}}{\left| x_{j}-x_{i}\right|}+\frac{1}{\left| x_{j}-x_{i}\right|}(\frac{x_{i}+x_{j}}{2}-\xi)\big\}\lambda_{x}^{ij}(d\xi)\nonumber\\
&=\int_{V_{i}(x)\cap V_{j}(x)}\varphi(\xi)\frac{1}{\left| x_{j}-x_{i}\right|}\big\{\frac{x_{j}-x_{i}}{2}+\frac{x_{i}+x_{j}}{2}-\xi\big\}\lambda_{x}^{ij}(d\xi)\nonumber\\
&=\int_{V_{i}(x)\cap V_{j}(x)}\varphi(\xi)\frac{1}{\left| x_{j}-x_{i}\right|}(x_{j}-\xi)\lambda_{x}^{ij}(d\xi).
\end{align}

\begin{proof}[Proof of Proposition \ref{contihg}]

{\color{black}

$(i)$ Set $\varphi^{i, M}(\xi)=(x_{i}-\xi)f(\xi)\chi_{M}(\xi)$ with 
\[\chi_{M}(\xi)\coloneqq \begin{cases}
1 & \quad\left|\xi\right|\leq M\\
M+1-\left|\xi\right| & \quad M<\left|\xi\right|\leq M+1\\
0&  \quad\left|\xi\right|> M+1\\
\end{cases}.\]
Set $\Phi_{i}^{M}(x)=\int_{V_{i}(x)}\varphi^{i, M}(\xi)d\xi$ and $\Phi_{i}(x)=\int_{V_{i}(x)}(x_{i}-\xi)f(\xi)d\xi$ for $i =1, ..., K$. Then (\ref{gradg}) implies that $\frac{\partial \mathcal{D}_{K,\,\mu}}{\partial x_{i}}=2\Phi_{i}, \; i=1, ..., K$.

%Lemma \ref{lemmajean}
For $j=1, ..., K$ and $j\neq i$, it follows from (\ref{simplifylemma})  that 
\begin{equation}\label{Mderiveeseconde1}
\frac{\partial \Phi_{i}^{M}}{\partial x_{j}}(x)=\int_{V_{i}(x)\cap V_{j}(x)}(x_{i}-\xi)\otimes(x_{j}-\xi)\cdot  \frac{1}{\left| x_{j}-x_{i}\right|}f(\xi)\chi_{M}(\xi)\lambda_{x}^{ij}(d\xi), 
\end{equation}
{\color{black}
and for $i=1,..., K$,
\small
\begin{equation}\label{Mderiveeseconde2}
\frac{\partial\Phi_{i}^{M}}{\partial x_{i}}(x)=\Big[\big(\int_{V_{i}(\xi)}f(\xi)\chi_{M}(\xi)d\xi\big)\mathrm{I}_{d} - \sum_{\substack{i\neq j \\ 1\leq j\leq K}}\int_{V_{i}(x)\cap V_{j}(x)}(x_{i}-\xi)\otimes(x_{i}-\xi)\cdot  \frac{1}{\left| x_{j}-x_{i}\right|}f(\xi)\chi_{M}(\xi)\lambda_{x}^{ij}(d\xi)\Big],
\end{equation}
}
where in (\ref{Mderiveeseconde1}) and (\ref{Mderiveeseconde2}), $u\otimes v\coloneqq[u^{i}v^{j}]_{1\leq i, j \leq d}$ for  any two vectors $u=(u^{1}, ..., u^{d})$ and $v=(v^{1}, ... , v^{d})$ in $\mathbb{R}^{d}$. 

\smallskip

We prove now the differentiability of $\Phi_{i}$ in three steps. 

\noindent \textit{$\blacktriangleright$ Step 1}: We prove in this part that for every $x\in F_{K}$, 
\[h_{ij}(x)\coloneqq\int_{V_{i}(x)\cap V_{j}(x)}(x_{i}-\xi)\otimes(x_{j}-\xi)\cdot  \frac{1}{\left| x_{j}-x_{i}\right|}f(\xi)\lambda_{x}^{ij}(d\xi)<+\infty. \]
If $V_{i}(x)\cap V_{j}(x)=\varnothing$, it is obvious that  $h_{ij}(x)=0<+\infty$. Now we assume that $V_{i}(x)\cap V_{j}(x)\neq\varnothing$. Without loss of generality, we assume that $V_{1}(x)\cap V_{2}(x)=\varnothing$ and we prove in the following $h_{12}$ is well defined i.e. $(h_{12}(x)\in\mathbb{R}$.

Let \begin{equation}\label{defalpha}
\alpha(x, \xi)\coloneqq (x_{1}-\xi)\otimes(x_{2}-\xi)\cdot  \frac{1}{\left| x_{2}-x_{1}\right|}f(\xi).
\end{equation}
Then 
\[h_{12}(x)=\int_{V_{1}(x)\cap V_{2}(x)}\alpha(x, \xi)\lambda_{x}^{12}(d\xi).\]

Let $(e_{1}, ..., e_{d})$ denote the canonical basis of $\mathbb{R}^{d}$. Set $u^{x}=\frac{x_{1}-x_{2}}{\left| x_{1}-x_{2}\right|}$. As $x_{1}\neq x_{2}$, there exists at least one $i_{0}\in\{1, ..., d\}$ s.t. $\langle u^{x}\mid e_{i_{0}}\rangle\neq0$. Then $(u^{x}, e_{i}, 1\leq i\leq d, i\neq i_{0})$ forms a new basis of $\mathbb{R}^{d}$. Applying the Gram-Schmidt orthonormalization procedure, we derive the existence of a new orthonormal basis $(u_{1}^{x}, ..., u_{d}^{x})$ of $\mathbb{R}^{d}$ such that $u_{1}^{x}=u^{x}$. Moreover, the Gram-Schmidt orthonormalization procedure also implies that $u_{i}^{x}, 1\leq i \leq d$ is continuous in $x$. With respect to this new basis $(u_{1}^{x}, ..., u_{d}^{x})$, the hyperplane $M_{12}^{x}$ defined in (\ref{midplan}) can be written by 
\[M_{12}^{x}=\frac{x_{1}+x_{2}}{2}+ \mathrm{span}\big( u_{i}^{x}, \; i=2, ..., d\big),\]
where $\mathrm{span}(S)$ denotes the vector subspace of $\mathbb{R}^{d}$ spanned by $S$. Moreover, note that
\[
V_{1}(x)\cap V_{2}(x)=\big\{ \xi \in M_{12}^{x} \;\big|\; \min_{k=3, ..., K}\left|x_{k}-\xi\right|\geq \left|x_{1}-\xi\right|=\left|x_{2}-\xi\right|\big\}.\]
Then, for every fixed $\xi \notin \partial \big(V_{1}(x)\cap V_{2}(x)\big)$, the function $x\mapsto \mathbbm{1}_{V_{1}(x)\cap V_{2}(x)}(\xi)$ is continuous in $x\in F_{K}$ and 
\begin{equation}\label{convexpolyhedral}
\lambda_{x}^{12}\Big(\partial \big(V_{1}(x)\cap V_{2}(x)\big)\Big)=0
\end{equation}
since $V_{1}(x)\cap V_{2}(x)$ is a polyhedral convex set in $M_{12}^{x}$.

Now by a change of variable $\xi=\frac{x_{1}+x_{2}}{2}+\sum_{i=2}^{d}r_{i}u_{i}^{x}$, 
\begin{equation}\label{changeg}
h_{12}(x)=\int_{\mathbb{R}^{d-1}}\mathbbm{1}_{V_{12}(x)}\big((r_{2}, ..., r_{d})\big)\alpha\Big(x, \frac{x_{1}+x_{2}}{2}+\sum_{i=2}^{d}r_{i}u_{i}^{x}\Big)dr_{2}...dr_{d},
\end{equation}
where
\begin{equation}\label{defv12}
V_{12}(x)\coloneqq \Big\{(r_{2}, ..., r_{d})\in\mathbb{R}^{d-1}\;\Big|\; \min_{3\leq k\leq K}\Big| x_{k}-\frac{x_{1}+x_{2}}{2}-\sum_{i=2}^{d}r_{i}u_{i}^{x}\Big|\geq \Big|\frac{x_{1}-x_{2}}{2}-\sum_{i=2}^{d}r_{i}u_{i}^{x}\Big|\Big\}.
\end{equation}
Let $\partial V_{12}(x)$ be the boundary of $V_{12}(x)$ given by
\begin{equation}\label{defparv12}
\partial V_{12}(x)\coloneqq \Big\{(r_{2}, ..., r_{d})\in\mathbb{R}^{d-1}\;\Big|\; \min_{3\leq k\leq K}\Big| x_{k}-\frac{x_{1}+x_{2}}{2}-\sum_{i=2}^{d}r_{i}u_{i}^{x}\Big|= \Big|\frac{x_{1}-x_{2}}{2}-\sum_{i=2}^{d}r_{i}u_{i}^{x}\Big|\Big\}.
\end{equation}
Then (\ref{convexpolyhedral}) implies that $\lambda_{\mathbb{R}^{d-1}}\big(\partial V_{12}(x)\big)=0$ where $\lambda_{\mathbb{R}^{d-1}}$ denotes the Lebesgue measure of the subspace $\mathrm{span}\big(u_{i}^{x}, i=2, ..., d\big)$.

It is obvious that for any $a=(a_{1}, ..., a_{d}), b=(b_{1}, ..., b_{d})\in\mathbb{R}^{d}$, we have $\left|a_{i}b_{j}\right|\leq \left|a\right|\left|b\right|, 1\leq i,j\leq d$. Thus the absolute value of every term in the matrix
\begin{align}\label{matrixalpha}
\alpha&(x, \frac{x_{1}+x_{2}}{2}+\sum_{i=2}^{d}r_{i}u_{i}^{x})\nonumber\\
&=\frac{\big(\frac{x_{1}-x_{2}}{2}-\sum_{i=2}^{d}r_{i}u_{i}^{x}\big)\otimes \big(\frac{x_{2}-x_{1}}{2}-\sum_{i=2}^{d}r_{i}u_{i}^{x}\big)}{\left| x_{2}-x_{1}\right|}f\Big(\frac{x_{1}+x_{2}}{2}+\sum_{i=2}^{d}r_{i}u_{i}^{x}\Big)
\end{align}
can be upper-bounded by 
\begin{align}\label{majo1}
&\frac{\big|\frac{x_{1}-x_{2}}{2}-\sum_{i=2}^{d}r_{i}u_{i}^{x}\big| \big|\frac{x_{2}-x_{1}}{2}-\sum_{i=2}^{d}r_{i}u_{i}^{x}\big|}{\left| x_{2}-x_{1}\right|}f\Big(\frac{x_{1}+x_{2}}{2}+\sum_{i=2}^{d}r_{i}u_{i}^{x}\Big)\nonumber\\
&\leq\frac{\Big(\big|\frac{x_{1}-x_{2}}{2}\big|+\big|\sum_{i=2}^{d}r_{i}u_{i}^{x}\big| \Big)^{2}}{\left| x_{2}-x_{1}\right|}f\Big(\frac{x_{1}+x_{2}}{2}+\sum_{i=2}^{d}r_{i}u_{i}^{x}\Big)\nonumber\\
&\leq C_{x}(1+\sum_{i=2}^{d}r_{i}^{2})f\Big(\frac{x_{1}+x_{2}}{2}+\sum_{i=2}^{d}r_{i}u_{i}^{x}\Big)
\end{align}
where $C_{x}>0$ is a constant depending only on $x$. 

The distribution $\mu$ is assumed to be $1$-radially controlled i.e. there 
 exist a constant $A>0$ and a continuous and decreasing function $g: \mathbb{R}_{+}\rightarrow \mathbb{R}_{+}$ such that 
\begin{equation}\label{dthcon}
\forall \xi\in\mathbb{R}^{d}, \left|\xi\right|\geq A,\;\;\;\;\;\;\;\;\; f(\xi)\leq g(\left|\xi\right|) \;\text{and}\;\int_{\mathbb{R}_{+}}x^{d}g(x)dx<+\infty.
\end{equation}
Now let $K\coloneqq \frac{1}{2}\left|x_{1}+x_{2}\right|\vee A$ and let $r\coloneqq \sum_{i=2}^{d}r_{i}u_{i}^{x}$. As $g$ is a non-increasing function, it follows that 
\begin{align}
&C_{x}(1+\sum_{i=2}^{d}r_{i}^{2})f\Big(\frac{x_{1}+x_{2}}{2}+\sum_{i=2}^{d}r_{i}u_{i}^{x}\Big)\nonumber\\
&\leq C_{x}(1+\left|r\right|^{2})\sup_{\xi\in B(\mathbf{0}, 3K)}f(\xi)\mathbbm{1}_{\{\left|r\right|\leq 2K\}}+C_{x}(1+\left|r\right|^{2})g\Big(\Big|\frac{x_{1}^{(n)}+x_{2}^{(n)}}{2}+\sum_{i=2}^{d}r_{i}u_{i}^{x}\Big|\Big)\mathbbm{1}_{\{\left|r\right|\geq 2K\}}. \nonumber\\
&\leq  C_{x}(1+\left|r\right|^{2})\sup_{\xi\in B(\mathbf{0}, 3K)}f(\xi)\mathbbm{1}_{\{\left|r\right|\leq 2K\}}+C_{x}(1+\left|r\right|^{2})g\big(\left|r\right|-K\big)\mathbbm{1}_{\{\left|r\right|\geq 2K\}}. 
\end{align}

%As $g$ is a decreasing function, we have 
%\[g\Big(\Big|\frac{x_{1}+x_{2}}{2}+\sum_{i=2}^{d}r_{i}u_{i}^{x}\Big|\Big)\mathbbm{1}_{\{\left|r\right|\geq 2K\}}\leq g\big(\left|r\right|-K\big)\mathbbm{1}_{\{\left|r\right|\geq 2K\}}.\]
Switching to polar coordinates, one obtains by letting $s=\left|r\right|$
\begin{align}
&\int_{\mathbb{R}^{d-1}}C_{x}\left|r\right|^{2}g\big(\left| r \right| -K\big)\mathbbm{1}_{\{\left|r\right|\geq 2K\}}dr_{2}...dr_{d}\nonumber\\
&\leq C_{x, d}\int_{\mathbb{R}_{+}}\; s^{2}g(s-K)\mathbbm{1}_{\{s\geq 2K\}}s^{d-2}ds\leq C_{x,d}\int_{K}^{\infty}(s+K)^{d}g(s)ds\nonumber\\
&\leq 2^{d}C_{x, d}\int_{K}^{\infty} (K^{d}+s^{d})g(s)ds<+\infty,\nonumber
\end{align}
where the last inequality follows from (\ref{dthcon}). Thus one obtains 
\[\int_{\mathbb{R}^{d-1}}\Big[ C_{x}(1+\left|r\right|^{2})\sup_{\xi\in B(\mathbf{0}, 3K)}f(\xi)\mathbbm{1}_{\{\left|r\right|\leq 2K\}}+C_{x}(1+\left|r\right|^{2})g\big(\left|r\right|-K\big)\mathbbm{1}_{\{\left|r\right|\geq 2K\}}\Big] dr_{2}...dr_{d}<+\infty.\]
Hence $h_{12}$ is well-defined since
\begin{equation}\label{bounded}
\int_{V_{1}(x)\cap V_{2}(x)}\left|\alpha(x, \xi)\right|\lambda_{x}^{12}(d\xi)<+\infty.
\end{equation}
%which implies that $g_{12}(x)<+\infty$.

\smallskip
\noindent \textit{$\blacktriangleright$ Step 2}: Now we prove that for any $x\in F_{K}$, 
\begin{equation}\label{valmatrix}
\sup_{y\in B(x,\, \varepsilon_{x})}\left|\frac{\partial \Phi_{i}^{M}}{\partial x_{j}}(y)- h_{ij}(y)\right| \xrightarrow{M\rightarrow+\infty}0, 
\end{equation}
where $\varepsilon_{x}=\frac{1}{3}\min_{1\leq i< j\leq K}\left|x_{i}-x_{j}\right|$ and (\ref{valmatrix}) means every term in the matrix converges to 0. 

First, for every fixed $y\in B(x, \varepsilon_{x})$, the absolute value of every term in the following matrix 
\begin{align}
&\frac{\partial \Phi_{i}^{M}}{\partial x_{j}}(y)-h_{ij}(y)=\int_{V_{i}(y)\cap V_{j}(y)}\frac{(y_{i}-\xi)\otimes(y_{j}-\xi)}{\left|y_{j}-y_{i}\right|}f(\xi)\big(1-\chi_{M}(\xi)\big)\lambda_{y}^{ij}(d\xi)\nonumber
\end{align}
can be upper bounded by 
\begin{equation}
f_{M}(y)\coloneqq\int_{V_{i}(y)\cap V_{j}(y)\cap \big(\mathbb{R}^{d}\setminus B(0, M+1)\big)}\frac{|y_{i}-\xi||y_{j}-\xi|}{\left|y_{j}-y_{i}\right|}f(\xi)\lambda_{y}^{ij}(d\xi).
\end{equation}
Moreover, the inequality (\ref{bounded}) implies that $f_{M}(y)$ converges to 0 for every $y\in B(x, \varepsilon_{x})$ as $M\rightarrow+\infty$. As $(f_{M})_{M}$ is  a monotonically decreasing sequence, one can obtain 
\[\sup_{y\in B(x, \varepsilon)}\big|f_{M}(y)\big|\rightarrow 0\]
owing to Dini's theorem, which in turn implies the convergence in (\ref{valmatrix}).

\smallskip
\noindent \textit{$\blacktriangleright$ Step 3}: It is obvious that $\Phi_{i}^{M}(x)$ converges to $\Phi_{i}(x)$ for every $x\in \mathbb{R}^{d}$ as $M\rightarrow+\infty$ since $\mu\in\mathcal{P}_{2}(\mathbb{R}^{d})$. Hence $\frac{\partial\Phi_{1}}{\partial x_{2}}(x)=h_{12}(x)$.  Then one can directly obtain (\ref{deriveeseconde1}) since $\frac{\partial \mathcal{D}_{K, \mu}}{\partial x_{j} x_{i}}=2 \frac{\partial \Phi_{i}}{\partial x_{j}}=2h_{ij}$ by applying (\ref{gradg}). The proof for (\ref{deriveeseconde2}) is similar.

\smallskip

\noindent$(ii)$ We will only prove the continuity of $\frac{\partial^{2}\mathcal{D}_{K, \mu}}{\partial x_{1}\partial x_{2}}$ and $\frac{\partial^{2}\mathcal{D}_{K, \mu}}{\partial x_{1}^{2}}$ at a point $x\in F_{K}$. The proof for $\frac{\partial^{2}\mathcal{D}_{K, \mu}}{\partial x_{i}\partial x_{j}}$ for  others $i, j\in\{1, ..., K\}$ is similar. We take the same definition of $\alpha(x, \xi)$ in (\ref{defalpha}), 
%\[\alpha(x, \xi)\coloneqq (x_{1}-\xi)\otimes(x_{2}-\xi)\cdot  \frac{1}{\left| x_{2}-x_{1}\right|}f(\xi).\]
then 
\[\frac{\partial^{2}\mathcal{D}_{K, \mu}}{\partial x_{1}\partial x_{2}}(x)=2\int_{V_{1}(x)\cap V_{2}(x)}\alpha(x, \xi)\lambda_{x}^{12}(d\xi)\]
and by the same change of variable (\ref{changeg}) as in $(i)$, we have 
\[\frac{\partial^{2}\mathcal{D}_{K, \mu}}{\partial x_{1}\partial x_{2}}(x)=2\int_{\mathbb{R}^{d-1}}\mathbbm{1}_{V_{12}(x)}\big((r_{2}, ..., r_{d})\big)\alpha\Big(x, \frac{x_{1}+x_{2}}{2}+\sum_{i=2}^{d}r_{i}u_{i}^{x}\Big)dr_{2}...dr_{d}\]
with the same definition of $V_{12}(x)$ as in (\ref{defv12}). }

{\color{black}

Let us now consider a sequence $x^{(n)}=(x_{1}^{(n)}, ..., x_{K}^{(n)})\in(\mathbb{R}^{d})^{K}$ converging to a point $x=(x_{1}, ..., x_{K})\in F_{K}$ satisfying that for every $n\in\mathbb{N}^{*}$,%For $n$ large enough such that 
\begin{equation}\label{deltax}
\big|x^{(n)}-x\big|\leq \delta_{x}\coloneqq\frac{1}{3}\min_{1\leq i,j\leq K, i\neq j} \left|x_{i}-x_{j}\right|,
\end{equation} 
so that $x^{(n)}\in F_{K}$ for every $n\in\mathbb{N}^{*}$. For a fixed $(r_{2}, ..., r_{d})\in \mathbb{R}^{d-1}$, the continuity of $x\mapsto \alpha(x, \frac{x_{1}+x_{2}}{2}+\sum_{i=2}^{d}r_{i}u_{i}^{x})$ in $F_{K}$ can be obtained by the continuity of $(x, \xi)\mapsto \alpha(x, \xi)$ and the continuity of Gram-Schmidt orthonormalization procedure. 

By the same reasoning as in (\ref{majo1}), the absolute value of every term in the matrix \[\alpha\Big(x^{(n)}, \frac{x_{1}^{(n)}+x_{2}^{(n)}}{2}+\sum_{i=2}^{d}r_{i}^{(n)}u_{i}^{x^{(n)}}\Big)\]
can be upper bounded by 
\[\frac{\Big(\big|\frac{x_{1}^{(n)}-x_{2}^{(n)}}{2}\big|+\big|\sum_{i=2}^{d}r_{i}u_{i}^{x^{(n)}}\big| \Big)^{2}}{\left| x_{2}^{(n)}-x_{1}^{(n)}\right|}f\Big(\frac{x_{1}^{(n)}+x_{2}^{(n)}}{2}+\sum_{i=2}^{d}r_{i}^{(n)}u_{i}^{x^{(n)}}\Big),\]
where there exists a constant $C_{x}$ depending only on $x$ such that 
\[\frac{\Big(\big|\frac{x_{1}^{(n)}-x_{2}^{(n)}}{2}\big|+\big|\sum_{i=2}^{d}r_{i}u_{i}^{x^{(n)}}\big| \Big)^{2}}{\left| x_{2}^{(n)}-x_{1}^{(n)}\right|}\leq C_{x}(1+\sum_{i=2}^{d}r_{i}^{2})\] 
since by (\ref{deltax}), one can get 
\[\forall n\in\mathbb{N}^{*}, \forall i,j\in\{1, ..., K\} \text{ with } i\neq j, \quad\delta_{x}\leq \left|x_{i}^{(n)}-x_{j}^{(n)}\right|\leq \max_{1\leq i,j\leq K}\left|x_{i}-x_{j}\right|+2\delta_{x}.\]

Moreover, if we take 
$K\coloneqq \frac{1}{2}\sup_{n}\left|x_{1}^{(n)}+x_{2}^{(n)}\right|\vee A$ and take $r_{n}\coloneqq \sum_{i=2}^{d}r_{i}u_{i}^{x^{(n)}}$, then 

%\vspace{1cm}
%[ PROBABLEMENT ERREUR ICI ]

%As $g$ is a decreasing function, it follows that 
\begin{align}
&C_{x}(1+\sum_{i=2}^{d}r_{i}^{2})f\Big(\frac{x_{1}^{(n)}+x_{2}^{(n)}}{2}+\sum_{i=2}^{d}r_{i}u_{i}^{x^{(n)}}\Big)\nonumber\\
&\leq C_{x}(1+\left|r\right|^{2})\sup_{\xi\in B(\mathbf{0}, 3K)}f(\xi)\mathbbm{1}_{\{\left|r\right|\leq 2K\}}+C_{x}(1+\left|r\right|^{2})g\Big(\Big|\frac{x_{1}^{(n)}+x_{2}^{(n)}}{2}+\sum_{i=2}^{d}r_{i}u_{i}^{x^{(n)}}\Big|\Big)\mathbbm{1}_{\{\left|r\right|\geq 2K\}}. \nonumber\\
&\leq  C_{x}(1+\left|r\right|^{2})\sup_{\xi\in B(\mathbf{0}, 3K)}f(\xi)\mathbbm{1}_{\{\left|r\right|\leq 2K\}}+C_{x}(1+\left|r\right|^{2})g\big(\left|r\right|-K\big)\mathbbm{1}_{\{\left|r\right|\geq 2K\}}.
\end{align}
%[ FIN ]
%As $g$ is a decreasing function, we have 
%\[g\Big(\Big|\frac{x_{1}+x_{2}}{2}+\sum_{i=2}^{d}r_{i}u_{i}^{x}\Big|\Big)\mathbbm{1}_{\{\left|r\right|\geq 2K\}}\leq g\big(\left|r\right|-K\big)\mathbbm{1}_{\{\left|r\right|\geq 2K\}}.\]
%%%%%%(26dec2019)Switching to polar coordinates, one obtains by letting $\gamma=\left|r\right|$
%%%%%%(26dec2019)\begin{align}
%%%%%%(26dec2019)&\int_{\mathbb{R}^{d-1}}C_{x}\left|r\right|^{2}g\big(\left| r \right| -K\big)\mathbbm{1}_{\{\left|r\right|\geq 2K\}}dr_{2}...dr_{d}\nonumber\\
%%%%%%(26dec2019)&\leq C_{x, d}\int_{\mathbb{R}_{+}}\; \gamma^{2}g(\gamma-K)\mathbbm{1}_{\{\gamma\geq 2K\}}\gamma^{d-2}d\gamma\leq C_{x,d}\int_{K}^{\infty}(\gamma+K)^{d}g(\gamma)d\gamma\nonumber\\
%%%%%%(26dec2019)&\leq 2^{d}C_{x, d}\int_{K}^{\infty} (K^{d}+%%%%%%(26dec2019)\gamma^{d})g(\gamma)d\gamma<+\infty,\nonumber
%%%%%%(26dec2019)\end{align}
%%%%%%(26dec2019)where the last inequality follows from (\ref{dthcon}). Thus one obtains 
By the same reasoning as in $(i)$-Step 1, we have
\[\int_{\mathbb{R}^{d-1}}\Big[ C_{x}(1+\left|r\right|^{2})\sup_{\xi\in B(\mathbf{0}, 3K)}f(\xi)\mathbbm{1}_{\{\left|r\right|\leq 2K\}}+C_{x}(1+\left|r\right|^{2})g\big(\left|r\right|-K\big)\mathbbm{1}_{\{\left|r\right|\geq 2K\}}\Big] dr_{2}...dr_{d}<+\infty,\]
which implies $\frac{\partial^{2}\mathcal{D}_{K, \mu}}{\partial x_{1}\partial x_{2}}(x^{(n)})\rightarrow \frac{\partial^{2}\mathcal{D}_{K, \mu}}{\partial x_{1}\partial x_{2}}(x)$ as $n\rightarrow+\infty$ by applying  Lebesgue's dominated convergence theorem. Thus $\frac{\partial^{2}\mathcal{D}_{K, \mu}}{\partial x_{1}\partial x_{2}}$ is continuous at $x\in F_{K}$.}
 %where $A$ is the constant in Definition \ref{deftail}-(i). 
%$C_{x}(1+\sum_{l=2}^{d}r_{l}^{2})f\Big(\frac{x_{1}+x_{2}}{2}+\sum_{i=2}^{d}r_{i}u_{i}^{x}\Big)$

It remains to prove the continuity of 
$x\mapsto \mu\big(V_{1}(x)\big)=\int_{\mathbb{R}^{d}}\mathbbm{1}_{V_{1}(x)}(\xi)f(\xi)\lambda_{d}(d\xi)$
to obtain the continuity of $\frac{\partial^{2}\mathcal{D}_{K, \mu}}{\partial x_{1}^{2}}$ defined in (\ref{deriveeseconde2}).  %The prove of the continuity of $\frac{\partial^{2}\mathcal{D}_{K, \mu}}{\partial x_{i}^{2}}, i=2, ..., K$ is similar. 
Remark that 
\[V_{1}(x)=\Big\{\xi\in\mathbb{R}^{d}\;\big|\; \left|\xi-x_{1}\right|\leq \min_{1\leq j\leq K}\left|\xi-x_{j}\right|\Big\},\]
and by \cite{graf2000foundations}[Proposition 1.3], 
\[\partial V_{1}(x)=\Big\{\xi\in\mathbb{R}^{d}\;\big|\;\left|\xi-x_{1}\right|= \min_{1\leq j\leq K}\left|\xi-x_{j}\right|\Big\}.\]
Then for any $\xi\notin\partial V_{1}(x)$, the function $x\mapsto\mathbbm{1}_{V_{1}(x)}(\xi)$ is continuous. As the norm $\left|\cdot\right|$ is the Euclidean norm, then $\lambda_{d}(\partial V_{i}(x))=0$ (see \cite{graf2000foundations}[Proposition 1.3 and Theorem 1.5]). For any $x\in F_{K}$ and a sequence $x^{(n)}$ converging to $x$, we have
$\mathbbm{1}_{V_{1}(x^{(n)})}(\xi)f(\xi)\leq f(\xi)\in L^{1}(\lambda_{d})$. Thus the continuity of $x\mapsto \mu\big(V_{1}(x)\big)=\int_{\mathbb{R}^{d}}\mathbbm{1}_{V_{1}(x)}(\xi)f(\xi)\lambda_{d}(d\xi)$ is a direct application of Lebesgue's dominated convergence theorem. 

\end{proof}

\subsection{Appendix D:  Proof of Proposition \ref{podedim1}}

\begin{proof}
(i) We will only deal with the uniform distribution $U([0,1])$.  The proof is similar for other uniform distributions.

In \cite{graf2000foundations}[see Example 4.17 and 5.5] and \cite{MR1664202}, the authors show that $\Gamma^{*}=\{\frac{2i-1}{2K}: i-1, ..., K\}$ is the unique optimal quantizers of $U([0,1])$. Let $x^{*}=(\frac{1}{2K}, ..., \frac{2i-1}{2K}, ..., \frac{2K-1}{2K})$, then one can compute explicitly $H_{\mathcal{D}}(x^{*})$:

\begin{equation}\label{hessiendim1}
H_{\mathcal{D}}(x^{*})=\left[\begin{array}{ccccccc}
\frac{3}{2K} & -\frac{1}{2K}& & & & & \Large0 \normalsize\\
 & \ddots & \ddots & \ddots\\
 &  & -\frac{1}{2K} & \frac{1}{K} & -\frac{1}{2K}\\
 &  &  & \ddots & \ddots & \ddots\\
 \Large 0\normalsize&  &  &  &  & -\frac{1}{2K} & \frac{3}{2K}
\end{array}\right],
\end{equation}

The matrix $H_{\mathcal{D}}(x^{*})$ is tridiagonal. If we denote by $f_{k}(x^{*})$ its $k$-th leading principal minor and we define $f_{0}(x^{*})=1$, then 
\begin{equation}\label{recuminor}
f_{k}(x^{*})=\frac{1}{K}f_{k-1}(x^{*})-\frac{1}{4K^{2}}f_{k-2}(x^{*}) \;\;\;\mathrm{for}\;k=2,..., K-1,
\end{equation} 
and $f_{1}(x^{*})=\frac{3}{2K}$ and $f_{K}(x^{*})=\left|H_{\mathcal{D}}(x^{*})\right|=\frac{3}{K}f_{K-1}(x^{*})-\frac{1}{4K^{2}}f_{K-2}(x^{*})$ (see \cite{el2003note}). One can solve from the three-term recurrence relation that 
\begin{align}\label{minor}
&f_{k}(x^{*})=\frac{2k+1}{2^{k}K^{k}}, \;\;\text{for}\;k=1,... , K-1 \\
\mathrm{And}\;\;\;& f_{K}(x^{*})=\frac{2{K}+1}{2^{K}K^{K}}+\frac{1}{2K}f_{K-1}. 
\end{align}
In fact, (\ref{minor}) is true for $k=1$. Suppose (\ref{minor}) holds for $k\leq K-2$, then owing to (\ref{recuminor})
\[f_{k+1}(x^{*})=\frac{1}{K}\cdot\frac{2k+1}{2^{k}K^{k}}-\frac{1}{4{K}^{2}}\cdot\frac{2(k-1)+1}{2^{k-1}K^{k-1}}=\frac{2(k+1)+1}{2^{k+1}K^{k+1}}.\]

Then it is obvious that $f_{k}(x^{*})>0$ for $k=1, ..., K$. Thus, $H_{\mathcal{D}}(x^{*})$ is positive definite.

\smallskip
\noindent(ii) We define for $i=2, ..., K$, $\widetilde{x}_{i}^{*}=\frac{x_{i-1}^{*}+x_{i}^{*}}{2}$, then the Voronoi region $V_{i}(x^{*})=[\widetilde{x}_{i}^{*},\; \widetilde{x}_{i+1}^{*}]$ for $i=2, ..., K-1$, $V_{1}(x^{*})=(-\infty, \widetilde{x}_{2}^{*}]$ and $V_{K}(x^{*})=[\widetilde{x}_{K}^{*}, +\infty)$. 

For $2\leq i \leq K-1$, 
\small
\begin{align}\label{lix}
L_{i}(x^{*})&=A_{i}-2B_{i-1, i}-2B_{i, i+1}\nonumber\\
&=2\mu\big(V_{i}(x^{*})\big)-(x^{*}_{i}-x_{i-1}^{*})f(\frac{x_{i-1}^{*}+x_{i}^{*}}{2})-(x^{*}_{i+1}-x_{i}^{*})f(\frac{x_{i}^{*}+x_{i+1}^{*}}{2})\nonumber\\
&=2\mu\big(V_{i}(x^{*})\big)-2(x_{i}^{*}-\widetilde{x}_{i}^{*})f(\widetilde{x}_{i}^{*})-{2}(\widetilde{x}_{i+1}^{*}-x_{i}^{*})f(\widetilde{x}_{i+1}^{*})\nonumber\\
&=\frac{2}{\mu\big(V_{i}(x^{*})\big)}\Big\{ \mu\big(V_{i}(x^{*})\big)^{2}-[x_{i}^{*}\mu\big(V_{i}(x^{*})\big)\nonumber\\
&\;\;\;\;\:\:-\widetilde{x}_{i}^{*}\mu\big(V_{i}(x^{*})\big)]f(\widetilde{x}_{i}^{*})- [\widetilde{x}_{i+1}^{*}\mu\big(V_{i}(x^{*})\big)-x_{i}^{*}\mu\big(V_{i}(x^{*})\big)]f(\widetilde{x}_{i+1}^{*})\Big\}\nonumber\\
&=\frac{2}{\mu\big(V_{i}(x^{*})\big)}\Big\{ \mu\big(V_{i}(x^{*})\big)^{2}-[\int_{V_{i}(x^{*})}\xi f(\xi)d\xi-\widetilde{x}_{i}^{*}\int_{V_{i}(x^{*})}f(\xi)d\xi]f(\widetilde{x}_{i}^{*})\nonumber\\
&\:\:\:\:\:\:\:- [\widetilde{x}_{i+1}^{*}\int_{V_{i}(x^{*})}f(\xi)d\xi-\int_{V_{i}(x^{*})}\xi f(\xi)d\xi]f(\widetilde{x}_{i+1}^{*})\Big\}\:\:\:\:\text{ \big(owing to (\ref{stationarych3})\big)}\nonumber\\
&=\frac{2}{\mu\big(V_{i}(x^{*})\big)}\Big\{\mu\big(V_{i}(x^{*})\big)^{2}-f(\widetilde{x}_{i}^{*})\int_{V_{i}(x^{*})}(\xi-\widetilde{x}_{i}^{*})f(\xi)d\xi+f(\widetilde{x}_{i+1}^{*})\int_{V_{i}(x^{*})}(\xi-\widetilde{x}_{i+1}^{*})f(\xi)d\xi\Big\}.
\end{align}

\normalsize
In order to study the positivity of $L_{i}(x^{*})$, we define a function $\varphi_{i}(u)$ for any $ i\in\{1, ...,K\}$ and for any $u=(u_{1}, ..., u_{K+1})\in F_{K+1}^{+}$ by
\begin{equation}\label{phi}
\varphi_{i}(u)\coloneqq\big[\int_{u_{i}}^{u_{i+1}}f(\xi)d\xi\big]^{2}-f(u_{i})\int_{u_{i}}^{u_{i+1}}(\xi-u_{i})f(\xi)d\xi+f(u_{i+1})\int_{u_{i}}^{u_{i+1}}(\xi-u_{i+1})f(\xi)d\xi, \;\;
\end{equation}
\phantom\qedhere
\end{proof}

\begin{lemma}\label{lemphi}
If $f$ is positive and differentiable and if $\log f$ is strictly concave, then for all $u=(u_{1}, ..., u_{K+1})\in F_{K+1}^{+}$, we have the following results for $\varphi_{i}(u)$ defined in (\ref{phi}),
\begin{enumerate}[(a)]
\item for every $i=1, ..., K$, $\varphi_{i}(u)>0$;
\item $\frac{\partial\varphi_{1}}{\partial u_{1}}(u)<0$;
\item $\frac{\partial\varphi_{K}}{\partial u_{K+1}}(u)>0$.
\end{enumerate}
\end{lemma}

\begin{proof}[Proof of lemma \ref{lemphi}]
For a fixed $i\in\{1, ..., K\}$, the partial derivatives of $\varphi_{i}$ are 
\begin{align}
&\frac{\partial\varphi_{i}}{\partial u_{i}}(u)=-2\big[\int_{u_{i}}^{u_{i+1}}f(\xi)d\xi\big]f(u_{i})- f'(u_{i})\int_{u_{i}}^{u_{i+1}}(\xi-u_{i})f(\xi)d\xi+f(u_{i})f(u_{i+1})(u_{i+1}-u_{i})\nonumber\\
&\frac{\partial\varphi_{i}\;\;\;}{\partial u_{i+1}}(u)=2\big[\int_{u_{i}}^{u_{i+1}}f(\xi)d\xi\big]f(u_{i+1})+f'(u_{i+1})\int_{u_{i}}^{u_{i+1}}(\xi-u_{i+1})f(\xi)d\xi\nonumber\\&\hspace{3cm}-f(u_{i})f(u_{i+1})(u_{i+1}-u_{i})\nonumber\\
&\frac{\partial\varphi_{i}}{\partial u_{l}}(u)=0, \:\:\:\text{for all}\;l\neq i \;\text{and} \;l\neq i+1.
\end{align}
The second derivatives of $\varphi_{i}$ are 
\begin{align}
&\frac{\partial^{2}\varphi_{i}}{\partial u_{i+1}\partial u_{i}}(u)=\frac{\partial^{2}\varphi_{i}}{\partial u_{i}\partial u_{i+1}}(u)
=-f(u_{i+1})f(u_{i})+(u_{i+1}-u_{i})\big(f(u_{i})f'(u_{i+1})-f'(u_{i})f(u_{i+1})\big)\nonumber\\
&\frac{\partial^{2}\varphi_{i}}{\partial u_{l}\partial u_{i}}(u)=\frac{\partial^{2}\varphi_{i}}{\partial u_{i}\partial u_{l}}(u)=0\:\:\:\:\text{for all}\;l\neq i\;\text{and}\;l\neq i+1.
\end{align}

If $\log f$ is strictly concave, then $\displaystyle(\log f)'=\frac{f'}{f}$ is strictly decreasing. For $u\in F_{K+1}^{+}$, we have $u_{i+1}>u_{i}$, then 
\[\frac{f'(u_{i+1})}{f(u_{i+1})} - \frac{f'(u_{i})}{f(u_{i})}= \frac{\:f'(u_{i+1})f(u_{i})-f(u_{i+1})f'(u_{i})\:}{f(u_{i})f(u_{i+1})}<0.\]
Thus $f'(u_{i+1})f(u_{i})-f(u_{i+1})f'(u_{i})<0$ and from which one can get $\displaystyle\frac{\partial^{2}\varphi_{i}}{\partial u_{i+1}\partial u_{i}}(u)<0$.

In fact, $\varphi_{i}$, $\frac{\partial\varphi_{i}}{\partial u_{i}}$, $\frac{\partial\varphi_{i}}{\partial u_{i+1}}$ and $\frac{\partial^{2}\varphi_{i}}{\partial u_{i+1}\partial u_{i}}$ only depend on the variables  $u_{i}$ and $u_{i+1}$.

\smallskip
\noindent(a) For $1\leq i\leq K$, $\varphi_{i}(u_{i+1}, u_{i+1})=0$. After the Mean value theorem, there exists  $\gamma\in(u_{i}, u_{i+1})$ such that
\begin{equation}\label{varphi1d}
\frac{1}{u_{i}-u_{i+1}}\big(\varphi_{i}(u_{i}, u_{i+1}) - \varphi_{i}(u_{i+1}, u_{i+1})\big)=\frac{\partial \varphi_{i}}{\partial u_{i}}(\gamma, u_{i+1}).
\end{equation}
Moreover, there exists  $\zeta\in(\gamma, u_{i+1})$ such that
\[\frac{1}{u_{i+1}-\gamma}\big(\frac{\partial\varphi_{i}}{\partial u_{i}}(\gamma, u_{i+1})-\frac{\partial\varphi_{i}}{\partial u_{i}}(\gamma, \gamma)\big) = \frac{\partial^{2}\varphi_{i}}{\partial u_{i+1}\partial u_{i}}(\gamma, \zeta).\]

As $\gamma<\zeta$, $\displaystyle\frac{\partial^{2}\varphi_{i}}{\partial u_{i+1}\partial u_{i}}(\gamma, \zeta)<0$. Thus $\displaystyle\frac{\partial\varphi_{i}}{\partial u_{i}}(\gamma, u_{i+1})<0$, since $\displaystyle\frac{\partial\varphi_{i}}{\partial u_{i}}(\gamma, \gamma)=0$. Then $\varphi_{i}(u_{i}, u_{i+1})>0$ by applying $\displaystyle\frac{\partial\varphi_{i}}{\partial u_{i}}(\gamma, u_{i+1})<0$ in $(\ref{varphi1d})$.

\smallskip
\noindent(b) After the Mean value theorem, there exists  $\gamma'\in(u_{1}, u_{2})$ such that 
\[\frac{\partial^{2}\varphi_{1}}{\partial u_{1}\partial u_{2}}(u_{1}, \gamma')=\frac{1}{u_{2}-u_{1}}\big(\frac{\partial\varphi_{1}}{\partial u_{1}}(u_{1}, u_{2})-\frac{\partial\varphi_{1}}{\partial u_{1}}(u_{1}, u_{1})\big).\]
As $\displaystyle\frac{\partial^{2}\varphi_{1}}{\partial u_{1}\partial u_{2}}(u_{1}, \gamma')<0$ and $\displaystyle\frac{\partial\varphi_{1}}{\partial u_{1}}(u_{1}, u_{1})=0$, one can get $\displaystyle\frac{\partial\varphi_{1}}{\partial u_{1}}(u_{1}, u_{2})<0$. 

\smallskip
\noindent(c) In the same way, there exists  $\zeta'\in(u_{K}, u_{K+1})$ such that 
\[\frac{\partial^{2}\varphi_{K}}{\partial u_{K}\partial u_{K+1}}(\zeta', u_{K+1})=\frac{1}{u_{K}-u_{K+1}}\big(\frac{\partial\varphi_{K}}{\partial u_{K+1}}(u_{K}, u_{K+1})-\frac{\partial\varphi_{K}}{\partial u_{K+1}}(u_{K+1}, u_{K+1})\big).\]
As $\displaystyle\frac{\partial^{2}\varphi_{K}}{\partial u_{K}\partial u_{K+1}}(\zeta', u_{K+1})<0$ and $\displaystyle\frac{\partial\varphi_{K}}{\partial u_{K+1}}(u_{K+1}, u_{K+1})=0$, one  gets $\displaystyle\frac{\partial\varphi_{K}}{\partial u_{K+1}}(u_{K}, u_{K+1})>0$. 
\end{proof}

\begin{proof}[Proof of Proposition \ref{podedim1}, continuation]

We set $\widetilde{x}^{*,M}\coloneqq(-M, \widetilde{x}^{*}_{2}, ..., \widetilde{x}^{*}_{K}, M)$ with  $M$ large enough such that $\widetilde{x}^{*,M}\in F_{K+1}^{+}$, then for $2\leq i \leq K-1$, $L_{i}(x^{*})=\frac{2}{\mu(V_{i}(x^{*}))}\varphi_{i}(\widetilde{x}^{*,M})$. Thus $L_{i}(x^{*})>0,\; i=2, ..., K-1$ owing to Lemma \ref{lemphi}-(a).

For $i =1$, 
\begin{align}
L_{1}(x^{*})&=A_{1}(x^{*})-2B_{1,2}(x^{*})\nonumber\\
&=\frac{2}{\mu\big(V_{1}(x^{*})\big)}\Big\{\mu\big(V_{1}(x^{*})\big)^{2}-f(\widetilde{x}^{*}_{2})\int_{V_{1}(x^{*})}(\widetilde{x}_{2}^{*}-\xi)f(\xi)d\xi\Big\}.\nonumber
\end{align}
If we denote $D_{1}(x^{*})\coloneqq\mu\big(V_{1}(x^{*})\big)^{2}-f(\widetilde{x}^{*}_{2})\int_{V_{1}(x^{*})}(\widetilde{x}_{2}^{*}-\xi)f(\xi)d\xi $, then 
\[D_{1}(x^{*})=\lim_{M\rightarrow+\infty}\varphi_{1}(\widetilde{x}^{*,M})+f(-M)\int_{V_{1}^{M}(x^{*})}\big(\xi-(-M)\big)f(\xi)d\xi ,\]
where $V_{1}^{M}(x^{*})=[-M, \widetilde{x}^{*}_{2}]$. 

For all $M$ such that $-M<\widetilde{x}^{*}_{2}$, $\displaystyle f(-M)\int_{V_{1}^{M}(x^{*})}\big(\xi-(-M)\big)f(\xi)d\xi >0$, then \[\lim_{M\rightarrow+\infty}f(-M)\int_{V_{1}^{M}(x^{*})}\big(\xi-(-M)\big)f(\xi)d\xi \geq0.\]
It follows from Lemma \ref{lemphi}-(b) that $\displaystyle\frac{\partial\varphi_{1}}{\partial u_{1}}(u)<0$ for $u\in F_{K+1}^{+}$, so that for a fixed $M_{1}$ such that $\widetilde{x}^{*, M_{1}}\in F_{K+1}^{+}$, we have $\displaystyle\varphi_{1}(\widetilde{x}^{*,M_{1}})\leq\lim_{M\rightarrow+\infty}\varphi_{1}(\widetilde{x}^{*,M})$. We also have $\varphi_{1}(\widetilde{x}^{*,M_{1}})>0$ by applying Lemma \ref{lemphi}-(a). It follows that 
\begin{align}
D_{1}(x^{*})&=\lim_{M\rightarrow+\infty}\varphi_{1}(\widetilde{x}^{*,M})+\lim_{M\rightarrow+\infty}f(-M)\int_{V_{1}^{M}(x^{*})}\big(\xi-(-M)\big)f(\xi)d\xi\nonumber\\
&\geq \varphi_{1}(\widetilde{x}^{*,M_{1}})+\lim_{M\rightarrow+\infty}f(-M)\int_{V_{1}^{M}(x^{*})}\big(\xi-(-M)\big)f(\xi)d\xi\nonumber\\
&>0.\nonumber
\end{align}
Then $L_{1}(x^{*})=\displaystyle\frac{2}{\mu\big(V_{1}(x^{*})\big)}D_{1}(x^{*})>0$. 

The proof of $L_{K}(x^{*})$ is similar by applying Lemma \ref{lemphi}-(c). Thus $H_{\mathcal{D}}(x^{*})$ is positive definite owing to Gershgorin circle theorem. 
\end{proof}

\subsection{Appendix E:  Proof of Theorem \ref{performance} - $(b)$ and $(c)$}
\begin{proof}[Proof of Theorem \ref{performance}]
\noindent(b) If $\mu$ has a $c$-th polynomial tail with $c>d+p$, then $\mu\in\mathcal{P}_{p}(\mathbb{R}^{d})$. Let $X, X_{1}, ..., X_{n}$ be i.i.d random variable with probability distribution $\mu$. Then,
\begin{align}
r_{n}&=\left\Vert R_{n}\right\Vert_{2}^{2}=\mathbb{E}\big[\max\,(\left| X_{1}\right|, ...,\left| X_{n}\right|)^{2}\big]=\mathbb{E}\big[\max(\left| X_{1}\right|^{p}, ...,\left| X_{n}\right|^{p})^{2/p}\big]\nonumber\\
&{\leq}\mathbb{E}\Big(\big[\sum_{i=1}^{n}\left|X_{i}\right|^{p}\big]^{2/p}\Big)\leq \Big[\mathbb{E}\big(\sum_{i=1}^{n}\left|X_{i}\right|^{p}\big)\Big]^{2/p}=\Big[n\,\mathbb{E}\left|X\right|^{p}\Big]^{2/p}=n^{2/p}\left\Vert X\right\Vert_{p}^{2},
\end{align}
where the last line is due to the fact that $X_{1}, ..., X_{n}$ have the same distribution as $X$. 
Moreover, we have 
\begin{equation}\label{50}
\rho_{K}(\mu)=K^{\frac{p+d}{d(c-p-d)}\gamma_{K}}\;\;\;\text{with}\;\;\;\lim_{K\rightarrow+\infty}\gamma_{K}=1
\end{equation}
owing to (\ref{polynomialtail}). It follows from (\ref{geneperformance}) that 
\[\mathbb{E}\big[\mathcal{D}(x^{(n)})-\mathcal{D}(x)\big]\leq\frac{2K}{\sqrt{n}}\Big[3r_{2n}^{2}+\big((2m_{2})\vee\rho_{K}(\mu)\big)\cdot\rho_{K}(\mu)\Big]\]
since $r_{2n}\geq m_{2}$ after the definitions of $r_{2n}$ and $m_{2}$. In addition, (\ref{50}) implies that $\rho_{K}(\mu)\rightarrow+\infty$ as $K\rightarrow +\infty$ and, for large enough $K$, $\rho_{K}(\mu)\geq 2m_{2}$. Therefore, 
\begin{align}
\mathbb{E}\big[\mathcal{D}(x^{(n)})-\mathcal{D}(x)\big]\leq& \frac{2K}{\sqrt{n}}\Big(3\cdot (2n)^{2/p}\left\Vert X\right\Vert_{p}^{2}+3K^{\frac{p+d}{d(c-p-d)}\gamma_{K}}\Big)\nonumber\\
=& \frac{K}{\sqrt{n}}\Big(C_{\mu, p}\,n^{2/p}+6K^{\frac{p+d}{d(c-p-d)}\gamma_{K}}\Big),\nonumber
\end{align}
where $C_{\mu, p}=6\cdot2^{2/p}\left\Vert X\right\Vert_{p}^{2}$ and $\lim_{K}\gamma_{K}=1$.

\noindent(c) The distribution $\mu$ is assumed to have a hyper-exponential tail, that is, $\mu=f\cdot \lambda_{d}$ with $f(\xi)=\tau \left| \xi\right|^{c}e^{-\vartheta\left|\xi\right|^{\kappa}}$ for $\left|\xi\right|$ large enough with $c>-d$. The real constant $\kappa$ is assumed to be greater than or equal to 2. Let $X$ be a random variable with probability distribution $\mu$. Therefore, for every $\lambda\in(0, \vartheta)$, $\mathbb{E}\,e^{\lambda\left|X\right|^{\kappa}}<+\infty$ and
\begin{align}\label{normrn}
r_{n}=&\left\Vert R_{n}\right\Vert_{2}^{2}=\mathbb{E}\big[\max(\left| X_{1}\right|, ...,\left| X_{n}\right|)^{2}\big]=\mathbb{E}\big[\max(\left| X_{1}\right|^{\kappa}, ...,\left| X_{n}\right|^{\kappa})^{2/\kappa}\big]\nonumber\\
=&\mathbb{E}\Big(\Big[\frac{1}{\lambda}\log\big(\max(e^{\lambda\left|X_{1}\right|^{\kappa}}, ...,e^{\lambda\left|X_{n}\right|^{\kappa}} )\big)\Big]^{2/\kappa}\Big)
\leq\left(\frac{1}{\lambda}\right)^{2/\kappa}\Big[\log\mathbb{E}\max(e^{\lambda\left|X_{1}\right|^{\kappa}}, ...,e^{\lambda\left|X_{n}\right|^{\kappa}})\Big]^{2/\kappa}\nonumber\\
\leq&\left(\frac{1}{\lambda}\right)^{2/\kappa}\Big\{\log\mathbb{E}\Big[\sum_{i=1}^{n}e^{\lambda\left|X_{i}\right|^{\kappa}}\Big]\Big\}^{2/\kappa}=\left(\frac{1}{\lambda}\right)^{2/\kappa}\big\{\log(n\mathbb{E}\,e^{\lambda\left|X\right|^{\kappa}})\big\}^{2/\kappa}\nonumber\\
=&\left(\frac{1}{\lambda}\right)^{2/\kappa}\big(\log \mathbb{E}\,e^{\lambda\left|X\right|^{\kappa}}+\log n\big)^{2/\kappa},
\end{align}
where the last line of (\ref{normrn}) is due to the fact that $X_{1}, ..., X_{n}$ have the same distribution than $X$. 
Under the same assumption as before, it follows from (\ref{upperbound}) that
%%%%%%%%5
%%%%%%%%
%%%%%%%%
\begin{equation}\label{boundrho}
\rho_{K}(\mu)\leq\gamma_{K}(\log K)^{1/\kappa}\cdot 2\vartheta^{-1/\kappa}\big(1+\frac{2}{d}\big)^{1/\kappa}\;\;\text{with}\;\;\limsup_{K\rightarrow+\infty} \gamma_{K}\leq 1.
\end{equation}
%by applying  (\ref{upperbound}). 
Moreover, it follows from (\ref{geneperformance}) that 
\[\mathbb{E}\big[\mathcal{D}(x^{(n)})-\mathcal{D}(x)\big]\leq\frac{2K}{\sqrt{n}}\Big[3r_{2n}^{2}+\big((2m_{2})\vee\rho_{K}(\mu)\big)\cdot\rho_{K}(\mu)\Big]\]
since $r_{2n}\geq m_{2}$ after the definitions of $r_{2n}$ and $m_{2}$. In addition, (\ref{boundrho}) implies that $\rho_{K}(\mu)\rightarrow+\infty$ as $K\rightarrow +\infty$ and, for large enough $K$, $\rho_{K}(\mu)\geq 2m_{2}$. Therefore, 
\begin{align}\label{trueforlambda}
\mathbb{E}\big[\mathcal{D}(x^{(n)})-\mathcal{D}(x)\big]\leq& \frac{2K}{\sqrt{n}}\Big\{3\cdot\Big( 1\vee \log\big(2\mathbb{E}\,e^{\lambda\left
|X\right|^{\kappa}}\big)\,\Big)^{2/\kappa}\big(\frac{1}{\lambda}\big)^{2/\kappa}\big[(\log n)^{2/\kappa}+1\big]\Big\}\nonumber\\
&+ 4\vartheta^{-2/\kappa}\gamma_{K}(\log K)^{2/\kappa} \big(1+\frac{2}{d}\big)^{2/\kappa}.
\end{align}
Inequality (\ref{trueforlambda}) is true for all $\lambda\in(0, \vartheta)$. We may take $\lambda=\frac{\vartheta}{2}$. It follows that 
\begin{equation}
\mathbb{E}\big[\mathcal{D}(x^{(n)})-\mathcal{D}(x)\big]\leq C_{\vartheta, \kappa, \mu}\cdot \frac{K}{\sqrt{n}}\Big[ 1+ (\log n)^{2/\kappa}+\gamma_{K}(\log K)^{2/\kappa}\big(1+\frac{2}{d}\big)^{2/\kappa}\Big], 
\end{equation}
where $C_{\vartheta, \kappa, \mu}=\Big[ 6\big(\frac{2}{\vartheta}\big)^{2/\kappa}\cdot(1\vee\log2\mathbb{E}\,e^{\vartheta\left|X\right|^{\kappa}/2})\Big]\vee 8\vartheta^{-2/\kappa}$ and $\limsup_{K}\gamma_{K}=1$.

Multi-dimensional normal distribution is a special case of hyper-exponential tail distribution, i.e. if $\mu=\mathcal{N}(m, \Sigma)$, we have $\kappa=2, \vartheta=\frac{1}{2}$ and $c=0$.  By the same reasoning as before, 
\[\mathbb{E}\big[\mathcal{D}(x^{(n)})-\mathcal{D}(x)\big]\leq C_{\mu}\cdot \frac{K}{\sqrt{n}}\Big[ 1+ \log n+\gamma_{K}\log K\big(1+\frac{2}{d}\big)\Big], \]
where $C_{\mu}=24\cdot\big(1\vee \log2\mathbb{E}\,e^{\left| X\right|^{2}/4}\big)$.  When $\mu=\mathcal{N}(0, \mathbf{I}_{d})$, $C_{\mu}=24(1+\frac{d}{2})\cdot\log 2$, since $\mathbb{E}\,e^{\left| X\right|^{2}/4}=2^{d/2}$ by the moment-generating function of a $\chi^{2}$ distribution. 
\end{proof}

%%%%%%%%%
%%%%%%%%%
%%%%%%%%%
%%%%%%%%%

\bibliographystyle{alpha}
\bibliography{cvgrate}

\end{document}